\documentclass{article}

\usepackage{amssymb}
\usepackage{amsmath}
\usepackage{amsthm}
\usepackage{amsfonts}
\usepackage[english]{babel}
\usepackage[utf8]{inputenc}

\usepackage[margin=3cm]{geometry}

\usepackage[unicode,colorlinks,plainpages=false,hyperindex=true,bookmarksnumbered=true,bookmarksopen=false,pdfpagelabels]{hyperref}
\hypersetup{urlcolor=cyan,linkcolor=blue,citecolor=red,colorlinks=true}
\usepackage{color}

\usepackage{tikz}   
\usetikzlibrary{arrows,arrows.meta}  

\makeatletter
\renewcommand*{\p@subsection}{\S\,}
\renewcommand*{\p@subsubsection}{\S\,}
\makeatother

\theoremstyle{plain}
\newtheorem{theorem}{Theorem}[section]
\newtheorem{lemma}[theorem]{Lemma}
\newtheorem{proposition}[theorem]{Proposition}

\theoremstyle{definition}
\newtheorem{definition}[theorem]{Definition}
\newtheorem{example}[theorem]{Example}
\newtheorem{problem}[theorem]{Problem}

\theoremstyle{remark}
\newtheorem{remark}[theorem]{Remark}

\numberwithin{equation}{section}

\newcommand{\N}{\ensuremath{\mathbb{N}}}

\newcommand{\Z}{\ensuremath{\mathbb{Z}}}
\newcommand{\kk}{\ensuremath{\Bbbk}}
\newcommand{\Xtt}{\ensuremath{\mathtt{X}}}

\newcommand{\tr}{\operatorname{tr}}

\newcommand{\Hom}{\operatorname{Hom}}

\newcommand{\Mat}{\operatorname{Mat}}
\newcommand{\id}{\operatorname{id}}
\newcommand{\Id}{\operatorname{Id}}
\newcommand{\Gl}{\operatorname{GL}}
\newcommand{\Rep}{\operatorname{Rep}}

\newcommand{\mult}{{\operatorname{m}}}
\newcommand{\wk}{{\operatorname{wk}}}
\newcommand{\Sym}{\operatorname{Sym}}
\newcommand{\ab}{\operatorname{ab}}

\newcommand{\delh}{\ensuremath{\hat{\partial}}}
\newcommand{\DDer}{\ensuremath{\mathbb{D}\text{er}}}

\newcommand\br[1]{\{ #1 \}}
\newcommand\dgal[1]{  \left\{\!\!\left\{#1\right\}\!\!\right\} }
\newcommand\dgsmall[1]{  \{\!\!\{#1\}\!\!\} }

\newcommand{\Addresses}{{
  \bigskip
  \footnotesize

\noindent M.~Fairon, \textsc{School of Mathematics and Statistics, University of Glasgow, University Place, Glasgow G12 8QQ, UK}

\noindent  \textit{E-mail address:}  \texttt{Maxime.Fairon@glasgow.ac.uk}

\noindent --------------, \textsc{Department of Mathematical Sciences, Schofield Building, Loughborough University, Epinal Way, Loughborough LE11 3TU, UK}

\noindent  \textit{E-mail address:}  \texttt{M.Fairon@lboro.ac.uk}

  \medskip

\noindent C.~McCulloch, \textsc{School of Mathematics and Statistics, University of Glasgow, University Place, Glasgow G12 8QQ, UK}

\noindent  \textit{E-mail address:}  \texttt{2366542m@student.gla.ac.uk}
}}

\title{Around Van den Bergh's double brackets\\for different bimodule structures}
\author{Maxime Fairon \& Colin McCulloch}

\date{\today}

\begin{document}

\maketitle

 \begin{abstract}
A double Poisson bracket, in the sense of M. Van den Bergh, is an operation on an associative algebra $A$ which induces a Poisson bracket on each representation space $\operatorname{Rep}(A,n)$ in an explicit way. 
In this note, we study the impact of changing the Leibniz rules underlying a double bracket. This change amounts to make a suitable choice of $A$-bimodule structure on $A\otimes A$.  
In the most important cases, we describe how the choice of $A$-bimodule structure fixes an analogue to Jacobi identity, and we obtain induced Poisson brackets on representation spaces. 
The present theory also encodes a formalisation of the widespread tensor notation used to write Poisson brackets of matrices in mathematical physics. 
 \end{abstract}


\section{Introduction} \label{S:Intro}

In an influential paper \cite{VdB1}, Van den Bergh introduced  the notion of \emph{double Poisson brackets}, which are non-commutative analogues of Poisson brackets. 
For a unital associative algebra $A$ over a field $\kk$ of characteristic $0$, a double bracket is a $\kk$-bilinear map $\dgal{-,-}:A\times A \to A \otimes A$ (where $\otimes:=\otimes_\kk$) satisfying the following rules for any $w,x,y,z\in A$: 
\begin{subequations} \label{Eq:Intro-dbr}
\begin{align}
  \dgal{x,y}=&-\tau_{(12)}\dgal{y,x}\,,   \label{Eq:Intro-dbr-cyc} \\
\dgal{x,yz}=&
(y\otimes 1) \dgal{x,z}
+ \dgal{x,y} (1\otimes z)\,,  \label{Eq:Intro-dbr-Der1} \\   
    \dgal{wx,y}=&
(1\otimes w) \dgal{x,y}
+ \dgal{w,y} (x\otimes 1)\,,  \label{Eq:Intro-dbr-Der2} 
\end{align}
\end{subequations}
where $\tau_{(12)} x\otimes y = y \otimes x$ and $(w\otimes y)(x\otimes z)= wx \otimes yz$. 
Furthermore, a double bracket is called Poisson if it satisfies some version of Jacobi identity valued in $A^{\otimes 3}$. 
Their importance is rooted in representation theory: if $A$ is equipped with a double Poisson bracket, then for each $n\geq 1$ the representation space $\Rep(A,n)$ parametrising representations of $A$ on $\kk^n$ inherits a Poisson bracket $\br{-,-}$. 
Furthermore, if we reinterpret \cite{VdB1}, we can characterise the Poisson bracket on $\Rep(A,n)$ in terms of an operation $\dgal{-,-}_{(n)}$ which is defined for two $n\times n$ matrix-valued functions $X,Y$ on  $\Rep(A,n)$ (e.g. matrices representing elements $x,y \in A$) through 
\begin{equation} \label{Eq:Br-VdB}
\dgal{X,Y}_{(n)}=
\sum_{1\leq i,j,k,l \leq n} 
\br{X_{ij},Y_{kl}} \, E_{kj}\otimes E_{il}\,.
\end{equation}
Here, $E_{ij}\in \Mat_{n\times n}(\kk)$ is the elementary matrix satisfying $(E_{ij})_{i'j'}=\delta_{ii'}\delta_{jj'}$. 
This operation thus takes two matrices as an input, and it yields (a linear combination of tensor products of) two matrices as an output. 
The advantage of this notation stems from translating the antisymmetry and Leibniz rules of the Poisson bracket $\br{-,-}$ in the form 
\begin{subequations} \label{Eq:Intro-VdB}
\begin{align}
  \dgal{X,Y}_{(n)}=&-\tau_{(12)}\dgal{Y,X}_{(n)}\,,   \label{Eq:Intro-cyc} \\
\dgal{X,YZ}_{(n)}=&
(Y\otimes \Id_n) \dgal{X,Z}_{(n)}
+ \dgal{X,Y}_{(n)} (\Id_n\otimes Z)\,,  \label{Eq:Intro-VdB1} \\   
    \dgal{XZ,Y}_{(n)}=&
(\Id_n\otimes X) \dgal{Z,Y}_{(n)}
+ \dgal{X,Y}_{(n)} (Z\otimes\Id_n)\,.  \label{Eq:Intro-VdB2} 
\end{align}
\end{subequations}
(Jacobi identity can also be recast into this matrix notation, but we shall not need it in the introduction.) 
A direct comparison of \eqref{Eq:Intro-dbr} and \eqref{Eq:Intro-VdB} thus motivates a double Poisson bracket as the natural formalisation of a family of Poisson brackets defined on $\Rep(A,n)$, $n\geq 1$. In particular, this principle emphasises an important reason for the development of the study of double Poisson brackets.  
Some other useful aspects of this theory include classification results \cite{B,ORS,ORS2,P16,VdW}, double Poisson cohomology \cite{AKKN,PV,VdW}, connection to (pre-)Calabi-Yau algebras \cite{BCS,FH,IKV,LV}, study from the point of view of properads \cite{L,LV} and relation to integrable systems \cite{CW,DSKV,F22,FV}.

\medskip

Let us now adopt a completely different perspective. 
In the field of integrable systems, 
an effective method used to perform computations involving a Poisson bracket $\br{-,-}$ on a space parametrised by matrices, such as $\Rep(A,n)$, consists in using a ``tensor notation" as follows.
We introduce an operation $\br{-\stackrel{\otimes}{,}-}$ such that, for two matrices $X,Y$ of size $n\times n$ which are seen as matrix-valued functions on the given space, we set (see e.g. \cite{BBT})
\begin{equation} \label{Eq:Br-tens}
\br{X\stackrel{\otimes}{,}Y}=
\sum_{1\leq i,j,k,l \leq n} 
\br{X_{ij},Y_{kl}} \, E_{ij}\otimes E_{kl}\,.
\end{equation}
Note the similarity with \eqref{Eq:Br-VdB}, but the different arrangement of the indices. In particular, the antisymmetry of the Poisson bracket on $\Rep(A,n)$ still takes the form $\br{X\stackrel{\otimes}{,}Y}=-\tau_{(12)} \br{Y\stackrel{\otimes}{,}X}$ as in \eqref{Eq:Intro-cyc}, but the Leibniz rules now imply 
\begin{subequations}
\begin{align}
    \br{X\stackrel{\otimes}{,}YZ}=&
(\Id_n\otimes Y) \br{X\stackrel{\otimes}{,}Z}
+ \br{X\stackrel{\otimes}{,}Y} (\Id_n\otimes Z)\,,
\label{Eq:Intro-tens1} \\   
    \br{XZ\stackrel{\otimes}{,}Y}=&
(X\otimes\Id_n) \br{Z\stackrel{\otimes}{,}Y}
+ \br{X\stackrel{\otimes}{,}Y} (Z\otimes\Id_n)\,. 
 \label{Eq:Intro-tens2} 
\end{align}
\end{subequations}
It is clear that, given a Poisson bracket on $\Rep(A,n)$, the notations $\dgal{-,-}_{(n)}$ and $\br{-\stackrel{\otimes}{,}-}$ can be used interchangeably (we only need to know $\br{X_{ij},Y_{kl}}$ to define both \eqref{Eq:Br-VdB} and \eqref{Eq:Br-tens}). 
However, if we want to formalise these two operations  \emph{at the level of the associative algebra} $A$ in the spirit of Van den Bergh \cite{VdB1}, such operations will no longer be equivalent. This idea is at the core of the present note, and the key motivation is that the different matrix Leibniz rules given in \eqref{Eq:Intro-VdB1}--\eqref{Eq:Intro-VdB2} and  \eqref{Eq:Intro-tens1}--\eqref{Eq:Intro-tens2}   have a non-commutative origin based on different $A$-bimodule structures on $A\otimes A$. 

\medskip 

In Section \ref{S:Def}, we start by introducing double brackets as $\kk$-bilinear maps $\dgal{-,-}:A\times A\to A\otimes A$ satisfying the cyclic antisymmetry  $\dgal{x,y}=-\tau_{(12)}\dgal{y,x}$ for all $x,y\in A$ together with Leibniz rules for both arguments, see Definition \ref{Def:Dbr}. Each Leibniz rule is determined by an $A$-bimodule structure on $A\otimes A$. Compatibility of these Leibniz rules with the cyclic antisymmetry puts some constraint on the choice of $A$-bimodule structures, which must be \emph{swap-commuting} in the sense of Definition \ref{Def:Sw-Bim}.  
Then, we define in \ref{ss:wDBR} the notion of double (weak) Poisson brackets, where we require the vanishing of a map $A^{\times 3}\to A^{\otimes 3}$ called the \emph{(weak) double Jacobiator}. This requirement should be seen as an analogue of Jacobi identity. While the notion of double Poisson brackets is due to Van den Bergh \cite{VdB1} for the case \eqref{Eq:Intro-dbr} associated with the outer bimodule structure, we are forced to consider weaker versions in order to encompass the formalisation of the operation \eqref{Eq:Br-tens} on associative algebras. We then move on to derive general properties of double brackets in \ref{ss:GenDBr}. We finish the section by defining morphisms and equivalences of double brackets which allow us to compare such operations even in the presence of different swap-commuting bimodule structures, see Figure \ref{fig:Equiv}.  

In Section \ref{S:Exmp}, we study four families of double brackets that are associated with the most important  swap-commuting bimodule structures on $A\otimes A$ (hence this fixes the Leibniz rules of the double brackets). 
The case originally considered by Van den Bergh \cite{VdB1} comes from the outer bimodule structure \eqref{Eq:Mod-out} which is studied in \ref{ss:E-Out}. The formalisation of the operation \eqref{Eq:Br-tens} comes from the right bimodule structure \eqref{Eq:Mod-R} which is studied in \ref{ss:E-Right}. We advise the reader to have a look at these examples while studying the general theory from Section \ref{S:Def}. In particular, it will be clear from Lemma \ref{Lem:OutJac}  that it is more natural to verify the vanishing of the double Jacobiator in the case associated with the  outer bimodule structure \eqref{Eq:Mod-out}, 
while we see from Lemma \ref{Lem:R-wJac} that  the vanishing of a \emph{weak} double Jacobiator is fundamental in the case associated with the right bimodule structure \eqref{Eq:Mod-R}. 

Following Van den Bergh \cite{VdB1}, we investigate as part of Section \ref{S:Rep} how to induce a bilinear operation on each representation space $\Rep(A,n)$, $n\geq 1$, when we have a double bracket on $A$. In fact, we consider the different families of double brackets from Section \ref{S:Exmp} because we can explicitly induce bilinear operations in their presence. 
We also explain how, in the cases of the outer and right bimodule structures, the induced operations are nothing else than the operations $\dgal{-,-}_{(n)}$ \eqref{Eq:Br-VdB} and $\br{-\stackrel{\otimes}{,}-}$ \eqref{Eq:Br-tens} mentioned above, which motivated the present work. 
We prove as part of Theorems \ref{Thm:Rep-Out-Lie} and \ref{Thm:Rep-Right-Lie} that these induced operations are Poisson brackets provided that the underlying double bracket on $A$ is (weak) Poisson. 

We conclude with Section \ref{S:Comp} where we compare the two most important classes of double brackets. Namely, the one associated with the outer bimodule structure \eqref{Eq:Mod-out} due to Van den Bergh \cite{VdB1}, and the one associated with the right bimodule structure \eqref{Eq:Mod-R} which formalises  the operation \eqref{Eq:Br-tens}. 
To help the reader understand how to perform computations with double brackets in more details, we start in Appendix  \ref{S:Case} the classification of gradient double Poisson brackets (associated with the outer bimodule structure, in the original sense of Van den Bergh \cite{VdB1}).

\medskip

Let us end by emphasising that our study consists of a modification of Van den Bergh's theory of double brackets in the presence of \emph{other} bimodule structures on $A\otimes A$. This differs from the double Poisson bimodules from \cite{CEEY} which are maps $A\times M \to (A\otimes M)\oplus (M\otimes A)$ where $M$ is a fixed $A$-bimodule.  
Another completely different point of view that was undertaken by Arthamonov \cite{Art}, then extended by Goncharov and Gubarev \cite{GG}, relies on a modification of the antisymmetry rule of a double bracket. We do not pursue these directions at all. 

\medskip 

\noindent \textbf{Acknowledgments.}
We thank D. Calaque, O. Chalykh, D. Fern\'{a}ndez, A. Pichereau, V. Rubtsov and D. Valeri for stimulating discussions about double brackets. 
The research of the first author was partly supported by a Rankin-Sneddon Research Fellowship of the University of Glasgow, and a Doctoral Prize Fellowship of Loughborough University. 
The second author was partly supported by a Summer Project Bursary of the University of Glasgow.


\section{A general theory of double brackets}  \label{S:Def}

We introduce a general definition of double brackets, before giving their main properties and ways to compare them. Many examples illustrating this theory can be found in Section \ref{S:Exmp}. 

\subsection{Notations} 

Throughout the paper, we use the following conventions. 
All algebras are finitely generated, unital and taken over a fixed field $\kk$ of characteristic $0$. 
Tensor products are not decorated when taken over the base field: $\otimes=\otimes_\kk$. 
For two algebras $A,B$, an element $d\in A\otimes B$ is denoted as $d'\otimes d''$ ($:=\sum_j d_j' \otimes d_j''$) using Sweedler's notation; no confusion should arise from this notation because all operations that we will use are $\kk$-linear. 
There is a left action of the symmetric group $S_n$ on $A^{\otimes n}$ through 
\begin{equation*}
    \tau:S_n \times A^{\otimes n} \to A^{\otimes n}\,, \quad 
    \tau(\sigma,a_1\otimes \ldots \otimes a_n)
    =\tau_\sigma(a_1\otimes \ldots \otimes a_n)
    =a_{\sigma^{-1}(1)}\otimes \ldots \otimes a_{\sigma^{-1}(n)}\,.
\end{equation*}
For $n=2$, $\tau_{(12)}$ swaps the two tensor factors of $A^{\otimes 2}$ and we denote this operation ${}^\circ : d \mapsto d^\circ=d'' \otimes d'$; we call the map ${}^{\circ}: A^{\otimes 2}\to A^{\otimes 2}$ the \emph{swap automorphism}.   
We will regularly see $A^{\otimes 2}$ as an algebra for the multiplication $(a_1\otimes b_1)(a_2\otimes b_2)=a_1a_2 \otimes b_1b_2$, for $a_1,a_2,b_1,b_2\in A$. 

The following $A$-bimodule structures on $A^{\otimes 2}$ will be used: for $a,b\in A$ and $d\in A^{\otimes 2}$, 
\begin{align}
&&a \cdot_l d \cdot_l b = ad'b \otimes d''\,,  \qquad  
&\text{(left bimodule structure)} ; \label{Eq:Mod-L} \\
&&a \cdot_r d \cdot_r b = d' \otimes ad''b\,,  \qquad  
&\text{(right bimodule structure)} ; \label{Eq:Mod-R}\\
&&a \cdot_{out} d \cdot_{out} b = ad' \otimes d'' b\,,  \qquad  
&\text{(outer bimodule structure)} ; \label{Eq:Mod-out}\\
&&a \cdot_{in} d \cdot_{in} b = d'b \otimes ad''\,,  \qquad  
&\text{(inner bimodule structure)} . \label{Eq:Mod-in}
\end{align}
Given two automorphisms $\alpha,\beta$ of $A$, we can modify a bimodule structure $\cdot$ on $A^{\otimes 2}$ into another bimodule, denoted by $\cdot^{\alpha,\beta}$, in the following way:  
\begin{equation}
a \cdot^{\alpha,\beta} d \cdot^{\alpha,\beta} b := \alpha(a)\cdot d \cdot \beta(b)\,,  \qquad  
\text{( }(\alpha,\beta)\text{-twisted bimodule structure)} . \label{Eq:Mod-Twis} 
\end{equation}
We denote $\cdot^{\alpha}:=\cdot^{\alpha,\alpha}$ the corresponding  \emph{$\alpha$-twisted} bimodule structure. For example, the left bimodule structure can be modified as follows: 
\begin{equation}
    a \cdot_l^{\alpha,\beta} d \cdot_l^{\alpha,\beta} b = \alpha(a)d'\beta(b) \otimes d''\,,  \qquad  
\text{( }(\alpha,\beta)\text{-twisted left bimodule structure)} . \label{Eq:Mod-Ltwis} 
\end{equation}
(The right, outer and inner bimodule structures can be twisted in the exact same way.) These twists amount to precomposing the bimodule multiplication with the map $\alpha\times \beta : A^{\times 2}\to A^{\times 2}$. 

Finally, $\N=\{0,1,2,\ldots\}$, $\Z=\{\ldots,-1,0,1,\ldots\}$, while $\N^\times$ and $\Z^\times$ denote these sets without $0$. 

\subsection{Double brackets} 

\begin{definition} \label{Def:Sw-Bim}
Fix a $A$-bimodule structure denoted $\cdot$ on $A^{\otimes 2}$. 
The \emph{swap $A$-bimodule structure} $\ast$ on $A^{\otimes 2}$ associated with $\cdot$ is defined for $a,b\in A$ and $d\in A^{\otimes 2}$ by 
\begin{equation} \label{Eq:M-swap}
a \ast d \ast b = (a\cdot d^\circ \cdot b)^\circ\,.
\end{equation}
We say that $\cdot$ is \emph{swap-commuting} if it commutes with its swap $A$-bimodule structure $\ast$; that is, for any $a_1,a_2,b_1,b_2\in A$ and $d\in A^{\otimes 2}$, we have  
\begin{equation} \label{Eq:M-swapCom}
    a_1\cdot (a_2\ast d \ast b_2) \cdot b_1 
    = a_2\ast (a_1\cdot d \cdot b_1) \ast b_2\,.
\end{equation}
\end{definition}
In the case of a swap-commuting bimodule structure, we can forget to write the braces in iterated bimodule multiplications such as \eqref{Eq:M-swapCom}. 
Note that a bimodule structure is swap-commuting if and only if its swap bimodule structure is swap-commuting. 

As examples of swap-commuting bimodules, we can consider the left bimodule structure \eqref{Eq:Mod-L} with swap given by the right bimodule structure \eqref{Eq:Mod-R}, or the outer bimodule structure \eqref{Eq:Mod-out} with swap given by the inner bimodule structure \eqref{Eq:Mod-in}, and vice-versa.
However, an arbitrary bimodule is not necessarily swap-commuting: given an anti-automorphism $\nu$ of $A$, the bimodule structure defined by 
\begin{equation}
    a \cdot d \cdot b = a d' \otimes \nu(b) d''\,, 
\end{equation}
does not necessarily commute with its swap bimodule structure. 

\begin{definition} \label{Def:Dbr}
Fix a swap-commuting $A$-bimodule structure $\cdot$ on $A^{\otimes 2}$. 
A \emph{double bracket (associated with $\cdot$)} on $A$ is a $\kk$-bilinear map 
$\dgal{-,-}: A\times A \to A \otimes A$ such that for any $a,b,c\in A$, 
\begin{subequations} \label{Eq:Dbr}
\begin{align}
  \dgal{a,b}=&-\dgal{b,a}^\circ\,,   \label{Eq:Dbr-cyc} \\
\dgal{a,bc}=&
b\cdot \dgal{a,c}
+ \dgal{a,b} \cdot c\,,  \label{Eq:Dbr-Der1} \\   
    \dgal{bc,a}=&
b\ast \dgal{c,a}
+ \dgal{b,a} \ast c\,,  \label{Eq:Dbr-Der2} 
\end{align}
\end{subequations}
where $\ast$ denotes the swap $A$-bimodule structure of $\cdot$.
\end{definition}

\begin{remark}
Van den Bergh's original definition  \cite{VdB1} corresponds to the outer bimodule $\cdot_{out}$ \eqref{Eq:Mod-out}.
\end{remark}

The properties \eqref{Eq:Dbr-cyc}, \eqref{Eq:Dbr-Der1} and \eqref{Eq:Dbr-Der2} of a double bracket are respectively called the (cyclic) antisymmetry, the right Leibniz rule, and the left Leibniz rule. 
It is important to note that the fixed bimodule structure $\cdot$ appears in the definition of the right Leibniz rule. Then, the Leibniz rules  \eqref{Eq:Dbr-Der1} and \eqref{Eq:Dbr-Der2} are equivalent due to the antisymmetry \eqref{Eq:Dbr-cyc},  by definition of $\ast$ as the swap $A$-bimodule structure of $\cdot$. 
Furthermore, the double bracket is well-defined as we can apply the Leibniz rules in any order because $\cdot$ is swap-commuting. 
Let us finally remark that Definition \eqref{Def:Dbr} can equivalently be restated in terms of a $\kk$-linear map $\dgal{-,-}:A\otimes A\to A \otimes A$, $a\otimes b \mapsto \dgal{a,b}$.

\begin{remark}
It is possible to work in even greater generalities by considering an arbitrary involution ${}^\circ : A^{\otimes 2}\to A^{\otimes 2}$, $d\mapsto d^\circ$, in the above definitions. Since all the examples presented in this paper rely on the involution ${}^\circ=\tau_{(12)}$, we do not consider this generalisation any further. 
\end{remark}

\subsection{Double (weak) Poisson brackets} \label{ss:wDBR}

Fix a double bracket $\dgal{-,-}$ on $A$. We can consider the map 
\begin{equation} \label{Eq:dbr-L}
\dgal{-,-}_L: A\times A^{\otimes 2} \to A^{\otimes 3}\,, \quad 
    \dgal{a,b_1\otimes b_2}_L=\dgal{a,b_1}\otimes b_2\,.
\end{equation}

\begin{definition} \label{Def:DPA}
The \emph{double Jacobiator} of $\dgal{-,-}$ is the map 
$\dgal{-,-,-}:A^{\times 3}\to A^{\otimes 3}$ that satisfies 
\begin{equation} \label{Eq:dJac}
    \dgal{a,b,c}=
\dgal{a, \dgal{b,c}}_L
+\tau_{(123)} \dgal{b, \dgal{c,a}}_L
+ \tau_{(132)} \dgal{c, \dgal{a,b}}_L \,.
\end{equation}
If the double Jacobiator identically vanishes, we say that $\dgal{-,-}$ is a \emph{double Poisson bracket}. We call the pair $(A,\dgal{-,-})$ a \emph{double Poisson algebra}. 
\end{definition}

The double Jacobiator is cyclically symmetric: 
\begin{equation}  \label{Eq:dJCyc}
\dgal{a,b,c}=\tau_{(123)} \dgal{b,c,a}=\tau_{(132)} \dgal{c,a,b}\,, \qquad \forall a,b,c\in A\,.
\end{equation}
In particular cases (depending on the underlying swap-commuting bimodule, see Section \ref{S:Exmp}), the double Jacobiator is a derivation in each argument for some specific bimodule structures on $A^{\otimes 3}$.  

\begin{remark}
A $\kk$-bilinear operation on $A$ (seen as a $\kk$-vector space) satisfying the cyclic antisymmetry \eqref{Eq:Dbr-cyc} and with vanishing double Jacobiator \eqref{Eq:dJac} is called a \emph{double Lie algebra}. 
Such structures have been independently introduced by Schedler \cite{Sch}, Odesskii, Rubtsov and Sokolov \cite{ORS2}, as well as De Sole, Kac and Valeri \cite{DSKV}.  
Therefore, our definition of double Poisson brackets extends the notion of a double Lie algebra in the presence of a multiplication on $A$ and a swap-commuting $A$-bimodule structure on $A^{\otimes 2}$. 
\end{remark}

While \eqref{Eq:dJCyc} tells us that the double Jacobiator satisfies 
$\dgal{-,-,-}=\tau_\sigma^{-1}\circ \dgal{-,-,-}\circ \tau_\sigma$ 
for $\sigma\in \{\id,(123),(132)\}$, this is not the case for the other permutations $\{(12),(13),(23)\}$. 

\begin{definition} \label{Def:wDPA} 
Let $\sigma\in \{(12),(13),(23)\}$. 
The $\sigma$\emph{-weak double Jacobiator} of $\dgal{-,-}$ is the map 
$\sigma\wk\!\dgal{-,-,-}:A^{\times 3}\to A^{\otimes 3}$ given by  
\begin{equation} \label{Eq:wdJac}
    \sigma\wk\!\dgal{-,-,-}=\dgal{-,-,-}\,-\,\tau_\sigma^{-1}\circ \dgal{-,-,-}\circ \tau_\sigma \,.
\end{equation}
If the $\sigma$-weak double Jacobiator identically vanishes, we say that $\dgal{-,-}$ is a \emph{double $\sigma$-weak Poisson bracket}. We call the pair $(A,\dgal{-,-})$ a \emph{double $\sigma$-weak Poisson algebra}. 
\end{definition}

In some cases, we will need to work in even greater generalities. 
\begin{definition} \label{Def:wDPA2} 
Let $\sigma,\sigma'\in \{(12),(13),(23)\}$. 
The $[\sigma,\sigma']$\emph{-weak double Jacobiator} of $\dgal{-,-}$ is the map 
$[\sigma,\sigma']\wk\!\dgal{-,-,-}:A^{\times 3}\to A^{\otimes 3}$ 
given by  
\begin{equation} \label{Eq:wdJac2}
    [\sigma,\sigma']\wk\!\dgal{-,-,-}=\dgal{-,-,-}\,-\,\tau_\sigma^{-1}\circ \dgal{-,-,-}\circ \tau_{\sigma'} \,.
\end{equation}
If the $[\sigma,\sigma']$-weak double Jacobiator identically vanishes, we say that $\dgal{-,-}$ is a \emph{double $[\sigma,\sigma']$-weak Poisson bracket}. We call the pair $(A,\dgal{-,-})$ a \emph{double $[\sigma,\sigma']$-weak Poisson algebra}. 
\end{definition}

Due to the definitions, it is obvious that a double Poisson bracket is a double $[\sigma,\sigma']$-weak Poisson bracket for any choice of  $\sigma,\sigma'\in \{(12),(13),(23)\}$ (hence $\sigma$-weak when $\sigma'=\sigma$). 
We also note that, due to \eqref{Eq:dJCyc} and $\tau_{(123)}^s\tau_\sigma=\tau_{\sigma}\tau_{(123)}^{-s}$, the $[\sigma,\sigma']$-weak double Jacobiator enjoys the cyclic symmetry\footnote{This is an important cyclicity property that we are willing to preserve. For that reason, Definition \ref{Def:wDPA2} does not encompass cases where $\sigma,\sigma'$ have different parity, such as $\sigma=(12)$ with $\sigma'=\id$.} 
\begin{equation}  \label{Eq:wdJCyc}
[\sigma,\sigma']\wk\!\dgal{a,b,c}
=\tau_{(123)} \,\, [\sigma,\sigma']\wk\!\dgal{b,c,a}
=\tau_{(132)} \,\, [\sigma,\sigma']\wk\!\dgal{c,a,b}\,, 
\qquad \forall a,b,c\in A\,.
\end{equation}
Moreover, we can get from \eqref{Eq:dJCyc} that the property of being 
$[\sigma,\sigma']$-weak Poisson, $[\sigma(123),(132)\sigma']$-weak Poisson, and $[\sigma(132),(123)\sigma']$-weak Poisson are equivalent. 
Therefore, it suffices to consider the following three groups of double (weak) Poisson brackets: 
\begin{itemize}
    \item double Poisson brackets, see Definition \ref{Def:DPA}; 
    \item double $(12)$-weak Poisson brackets, see Definition \ref{Def:wDPA};
    \item double $[(12),(13)]$-weak or $[(12),(23)]$-weak Poisson brackets, see Definition \ref{Def:wDPA2}.
\end{itemize}
These different cases will appear naturally as part of the families of examples considered in Section \ref{S:Exmp}. 

\subsection{General properties} \label{ss:GenDBr}

Fix a swap-commuting $A$-bimodule structure $\cdot$ on $A^{\otimes 2}$, and let $\dgal{-,-}$ be a double bracket associated with that bimodule. 
 For the next results, we consider a set of generators $S=\{a_j \mid j\in J\}$ for the algebra $A$.
\begin{lemma} \label{Lem:Gen}
The double bracket is completely determined by its value on the elements of $S$. 
Furthermore, if $<$ is a total order on $J$, the double bracket is completely determined by the elements  $\dgal{a_i,a_j}\in A^{\otimes 2}$ for $i<j$ in $J$, 
and  $\dgal{a_i,a_i}\in A^{\otimes 2}$ for $i\in J$.
\end{lemma}
\begin{proof}
It follows from the Leibniz rules that for 
$a=a_{i_1}\ldots a_{i_m}\in A$ and  
$b=a_{j_1}\ldots a_{j_n}\in A$, 
with $i_1,\ldots,i_m,j_1,\ldots,j_n\in J$, we have 
\begin{equation*}
\dgal{a,b}=\sum_{k=1}^{m} \sum_{l=1}^{n}
a_{i_1}\ldots a_{i_{k-1}} \ast a_{j_1}  \ldots a_{j_{l-1}} \cdot 
\dgal{a_{i_k} , a_{j_l} }
\cdot a_{j_{l+1}}\ldots a_{j_n} \ast a_{i_{k+1}}\ldots a_{i_m}\,.
\end{equation*}
(We follow the convention that any empty product equals $1$.) 
Any two elements in $A$ are sums of terms of that form, so we can conclude by $\kk$-bilinearity. 

For the second part, the cyclic antisymmetry \eqref{Eq:Dbr-cyc} allows to get  $\dgal{a_j,a_i}=-\dgal{a_i,a_j}^\circ$ for $i<j$. Hence the value of the double bracket is known on the elements of $S$, and we can use the first part.  
\end{proof}

As a consequence of Lemma \ref{Lem:Gen}, we can easily construct double brackets. 

\begin{example} \label{Exmp:Free2}
On $A=\kk\langle x^{\pm 1}, y^{\pm1}\rangle$, fix $d_1,d_2,d_3\in A^{\otimes 2}$. Then the values 
\begin{equation*}
    \dgal{x,y}=d_1\,, \quad \dgal{x,x}=d_2 - d_2^\circ\,, \quad 
    \dgal{y,y}=d_3 - d_3^\circ\,,
\end{equation*}
completely characterise a double bracket on $A$ (for any fixed $A$-bimodule structure $\cdot$ on $A^{\otimes 2}$). 
We have $\dgal{y,x}=-d_1^\circ$ and, due to $\kk$-bilinearity, we can find the value of the double bracket on inverses, e.g.  
\begin{equation*}
    \dgal{a,b^{-1}}=-b^{-1} \cdot \dgal{a,b}\cdot b^{-1}\,, \quad 
    \dgal{a^{-1},b}=-a^{-1} \ast \dgal{a,b}\ast a^{-1}\,.
\end{equation*}
\end{example}

The next lemma (and its weak version, Lemma \ref{Lem:wJac}) will play an important role in Section \ref{S:Exmp}. 

\begin{lemma} \label{Lem:Jac}
Assume that there exists a $A$-bimodule structure on $A^{\otimes 3}$ for which  
the double Jacobiator \eqref{Eq:dJac} is a derivation in one of its arguments.  
Then it is a derivation in all of its arguments. Furthermore, 
the double bracket is Poisson if and only if the double Jacobiator vanishes on any $3$ elements of $S$. 
\end{lemma}
\begin{proof}
Let us assume that the assumption is satisfied for the third argument, 
i.e. there exists a $A$-bimodule on $A^{\otimes 3}$, denoted $\cdot_3$, such that 
\begin{equation*}
    \dgal{a,b,c_1c_2}=c_1\cdot_3 \dgal{a,b,c_2} + \dgal{a,b,c_1} \cdot_3 c_2\,.
\end{equation*}
Using the cyclic symmetry \eqref{Eq:dJCyc}, we get derivation rules in the other arguments: 
\begin{align*}
    \dgal{a,b_1b_2,c}&=b_1\cdot_2 \dgal{a,b_2,c} + \dgal{a,b_1,c} \cdot_2 b_2\,, \\
    \dgal{a_1a_2,b,c}&=a_1\cdot_1 \dgal{a_2,b,c} + \dgal{a_1,b,c} \cdot_1 c_2\,,
\end{align*}
where we define the $A$-bimodules $\cdot_{1,2}$ on $A^{\otimes 3}$   by 
$$\cdot_i=\tau_{(132)}\circ  \cdot_{i+1} \circ  (\id_A\times \tau_{(123)}\times \id_A) \,:\, 
A\times A^{\otimes 3}\times A\to A^{\otimes 3}\,, \qquad i=1,2\,.$$ 
For the second part of the statement, the forward part is obvious. For the backward part, we can write any double Jacobiator as a sum over $j,k,l\in S$ whose terms are obtained by applying some bimodule multiplications $\cdot_i$, $1\leq i \leq 3$, to $\dgal{a_j,a_k,a_l}$. 
If each $\dgal{a_j,a_k,a_l}=0$, any double Jacobiator vanishes. 
This proof is easily adapted if we start with a derivation in the other two arguments. 
\end{proof}

\begin{lemma} \label{Lem:Jac2}
Under the assumptions of Lemma \ref{Lem:Jac}, 
if the double bracket takes constant values on the elements of $S$ (i.e. $\dgal{a_i,a_j}=\lambda_{ij}\, 1\otimes 1$ for $\lambda_{ij}\in \kk$, with $\lambda_{ij}=-\lambda_{ji}$), then it is Poisson. 
\end{lemma}
\begin{proof}
By $\kk$-bilinearity, $\dgal{a_j, \dgal{a_k,a_l}}_L=\lambda_{kl} \, \dgal{a_j,1}\otimes 1=0$. Thus,  the double Jacobiator vanishes on any triple from $S$ by \eqref{Eq:dJac}, and we can conclude using Lemma \ref{Lem:Jac}. 
\end{proof}

\begin{example}
All double brackets with constant values on generators of $A=\kk\langle x^{\pm 1}, y^{\pm1}\rangle$ (which arise as a special case of Example \ref{Exmp:Free2}) must satisfy  
$\dgal{x,y}=\lambda \, 1\otimes 1$ for $\lambda\in \kk$ and $\dgal{x,x}=0=\dgal{y,y}$. 
Provided that the double Jacobiator is a derivation in its third argument, these are double Poisson brackets by Lemma \ref{Lem:Jac2}. 
\end{example}

We finish by giving equivalent definitions of the double Jacobiator \eqref{Eq:dJac}. To this end, we introduce in analogy with \eqref{Eq:dbr-L} the following maps 
\begin{subequations}
\begin{align}
&\dgal{-,-}_R: A\times A^{\otimes 2} \to A^{\otimes 3}\,, \quad 
    \dgal{a,b_1\otimes b_2}_R=b_1\otimes \dgal{a,b_2}\,, \\
&\dgal{-,-}_L: A^{\otimes 2}\times A \to A^{\otimes 3}\,, \quad 
\dgal{a_1\otimes a_2,b}_L=\dgal{a_1,b}'\otimes a_2 \otimes \dgal{a_1,b}''\,, \\
&\dgal{-,-}_R: A^{\otimes 2}\times A \to A^{\otimes 3}\,, \quad 
\dgal{a_1\otimes a_2,b}_R=a_1\otimes \dgal{a_2,b} \,. 
\label{Eq:RightDmap}
\end{align}
\end{subequations}
(Our convention for $\dgal{a_1\otimes a_2,b}_R$ is \emph{different} from the one in \cite{DSKV}.)

\begin{lemma}
For any $a,b,c\in A$, we have that 
\begin{align}
    \dgal{a,b,c}&=
\dgal{a, \dgal{b,c}}_L
- \dgal{b, \dgal{c,a}}_R
- \dgal{\dgal{a,b},c}_L\,, \label{Eq:dJac-DSKV} \\
\dgal{a,b,c}&=
-\dgal{b, \dgal{a,c}}_R
-\tau_{(123)} \dgal{c, \dgal{b,a}}_R
-\tau_{(132)} \dgal{a,\dgal{c,b}}_R\,, \label{Eq:dJac-R}\\
\dgal{a,b,c}&=
\tau_{(12)}\left( 
\dgal{\dgal{b,a},c}_R
+\tau_{(123)} \dgal{\dgal{a,c},b}_R
+\tau_{(132)} \dgal{\dgal{c,b},a}_R\right)
\, \label{Eq:dJac-InR}\,.
\end{align}
Equivalently, we can write \eqref{Eq:dJac-R} and \eqref{Eq:dJac-InR} as  
\begin{align*}
\dgal{-,-,-}&=-\sum_{s=0,1,2}\, \tau_{(123)}^s \circ \dgal{-,\dgal{-,-}}_R\circ \tau_{(12)}\circ \tau_{(123)}^{-s}\,, \\
\dgal{-,-,-}&=
\sum_{s=0,1,2}\, \tau_{(12)}\circ \tau_{(123)}^s \circ \dgal{\dgal{-,-},-}_R\circ \tau_{(123)}^{-s}\circ \tau_{(12)}
\,.
\end{align*}
\end{lemma}
\begin{proof}
Equation \eqref{Eq:dJac-DSKV} is Remark 2.2 in \cite{DSKV}. Indeed, the definition of the double Jacobiator \eqref{Eq:dJac} does not require to know the underlying swap-commuting $A$-bimodule structure. 

For \eqref{Eq:dJac-R} this follows from $\dgal{a,\dgal{b,c}}_L=-\tau_{(132)}\dgal{a,\dgal{c,b}}_R$. For  \eqref{Eq:dJac-InR}, it is a consequence of  $\dgal{a,\dgal{b,c}}_L=\tau_{(13)}\dgal{\dgal{c,b},a}_R$. 
\end{proof}

Similar statements hold in the weak case. We only state the following two results (the $\sigma$-weak case follows by taking $\sigma'=\sigma$), which are easily adapted from Lemmas \ref{Lem:Jac} and \ref{Lem:Jac2} using the cyclic symmetry \eqref{Eq:wdJCyc} of the $[\sigma,\sigma']$-weak double Jacobiator.  
\begin{lemma} \label{Lem:wJac}
Let $\sigma,\sigma'\in \{(12),(13),(23)\}$. 
Assume that there exists a $A$-bimodule structure on $A^{\otimes 3}$ for which  
the $[\sigma,\sigma']$-weak double Jacobiator \eqref{Eq:wdJac2} is a derivation in one of its arguments. 
Then it is a derivation in all of its arguments. Furthermore, 
the double bracket is $[\sigma,\sigma']$-weak Poisson if and only if the $[\sigma,\sigma']$-weak double Jacobiator vanishes on any $3$ elements of $S$. 
\end{lemma}
\begin{lemma} \label{Lem:wJac2}
Under the assumptions of Lemma \ref{Lem:wJac}, 
if the double bracket takes constant values on the elements of $S$, then it is $[\sigma,\sigma']$-weak Poisson. 
\end{lemma}

\subsection{Morphisms and equivalences}  \label{ss:MorEq}

For $i=1,2$, let $A_i$ be equipped with a double bracket $\dgal{-,-}_i$ associated with a swap-commuting bimodule structure $\cdot_i$ on $A_i^{\otimes 2}$. 

\begin{definition} \label{Def:Morph}
A \emph{morphism of double brackets} is an algebra homomorphism $\phi:A_1 \to A_2$  such that  
\begin{equation} \label{Eq:morph}
    \dgal{\phi(a),\phi(b)}_2 = (\phi\otimes \phi)(\dgal{a,b}_1)\,, \qquad 
    \forall a,b\in A_1\,.
\end{equation}
If both double brackets $\dgal{-,-}_i$ are Poisson (resp. $\sigma$-weak Poisson or $[\sigma,\sigma']$-weak Poisson), we say that $\phi$ is a \emph{morphism of double Poisson algebras} 
(resp. \emph{of double $\sigma$-weak Poisson algebras} or \emph{of double $[\sigma,\sigma']$-weak Poisson algebras}).  
\end{definition}

In the presence of a morphism of double brackets, the two bimodule structures could be unrelated. Indeed, any algebra homomorphism  is a morphism of double Poisson algebras if we take the zero double bracket on both algebras (for  \emph{any two} swap-commuting bimodule structures). 
Nevertheless, the bimodule structures are not arbitrary in general, since the Leibniz rules yield identities such as 
\begin{equation*}
    \phi(b)\cdot_2 \dgal{\phi(a),\phi(c)}_2
+ \dgal{\phi(a),\phi(b)}_2 \cdot_2 \phi(c) 
= (\phi\otimes \phi)(b\cdot_1\dgal{a,c}_1)
+(\phi\otimes \phi)(\dgal{a,b}_1\cdot_1 c)\,, \quad a,b,c\in A_1\,.
\end{equation*}
In particular, we can consider morphisms when the two double brackets are associated with the left bimodule structure \eqref{Eq:Mod-L} on $A_1$ and $A_2$ respectively. This observation can also be made for the right  \eqref{Eq:Mod-R}, outer \eqref{Eq:Mod-out} or inner \eqref{Eq:Mod-in} bimodule structures, see  \ref{ss:Exmp-Morph} for some examples. 

A morphism of double brackets transforms the double Jacobiators \eqref{Eq:dJac} according to  
\begin{equation} \label{Eq:Mor-dJac}
    \dgal{\phi(a),\phi(b),\phi(c)}_2 = \phi^{\otimes 3}(\dgal{a,b,c}_1)\,, \qquad a,b,c\in A_1\,.
\end{equation} 
As a consequence, we get the following generalisations of \cite[Lemma 3.2]{F22}.

\begin{lemma} \label{Lem:Morph-Poiss}
Let  $\phi:A_1 \to A_2$ be a morphism of double brackets. 

\noindent If $\phi$ is surjective as an algebra homorphism and $A_1$ is a double Poisson algebra, then $A_2$ is a double Poisson algebra.
 If $\phi$ is injective as an algebra homorphism and $A_2$ is a double Poisson algebra, then $A_1$ is a double Poisson algebra. 
 In both cases, $\phi$ is a morphism of double Poisson algebras. 
\end{lemma}

\begin{lemma} \label{Lem:Morph-wkPoiss}
Let  $\phi:A_1 \to A_2$ be a morphism of double brackets, and $\sigma,\sigma'\in \{(12),(13),(23)\}$. 

\noindent If $\phi$ is surjective as an algebra homorphism and $A_1$ is a double $[\sigma,\sigma']$-weak Poisson algebra, then $A_2$ is a double $[\sigma,\sigma']$-weak Poisson algebra.
 If $\phi$ is injective as an algebra homorphism and $A_2$ is a double $[\sigma,\sigma']$-weak Poisson algebra, then $A_1$ is a double $[\sigma,\sigma']$-weak Poisson algebra. 
 In both cases, $\phi$ is a morphism of double $[\sigma,\sigma']$-weak Poisson\footnote{All statements work for the $\sigma$-weak case by taking $\sigma'=\sigma$.} algebras. 
\end{lemma}
\begin{proof}
Knowing the transformation of the double Jacobiator \eqref{Eq:Mor-dJac}, we have by \eqref{Eq:wdJac2} the same transformation for the $[\sigma,\sigma']$-weak double Jacobiator: 
$[\sigma,\sigma']\wk\!\dgal{\phi(a),\phi(b),\phi(c)}_2 
= \phi^{\otimes 3}([\sigma,\sigma']\wk\!\dgal{a,b,c}_1)$. 
So this follows by a straightforward adaptation of \cite[Lemma 3.2]{F22}. 
\end{proof}

\begin{definition} \label{Def:wMorph}
A \emph{warped morphism of double brackets} is an algebra homomorphism $\psi:A_1 \to A_2$ such that there exists a morphism $\Psi:A_1^{\otimes 2}\to A_2^{\otimes 2}$ that intertwines the bimodule structures, i.e. 
\begin{equation} \label{Eq:wBimod}
\Psi(a \cdot_1 d \cdot_1 b) = \psi(a) \cdot_2 \Psi(d) \cdot_2 \psi(b)\,, \qquad 
a,b\in A_1,\,\, d\in A_1^{\otimes 2}\,,
\end{equation}
in such a way that $\Psi(\dgal{a,b}_1)=\dgal{\psi(a),\psi(b)}_2$.  
\end{definition}

A morphism of double brackets is easily seen to correspond to the case of a warped morphism of double brackets  where  $\Psi=\psi\otimes \psi$. 
If we denote by $\ast_i$ the swap bimodule of $\cdot_i$, $i=1,2$, then the analogue of identity \eqref{Eq:wBimod} with the corresponding swap bimodule structures is 
\begin{equation} \label{Eq:wBimod-Swap}
\widetilde{\Psi}(a \ast_1 d \ast_1 b) = \psi(a) \ast_2 \widetilde{\Psi}(d) \ast_2 \psi(b)\,, \qquad 
a,b\in A_1,\,\, d\in A_1^{\otimes 2}\,,
\end{equation}
where $\widetilde{\Psi}:A_1^{\otimes 2}\to A_2^{\otimes 2}$, $\widetilde{\Psi}(d)=\Psi(d^\circ)^\circ$. 

\begin{definition}
Assume that $A:=A_1=A_2$. 
An \emph{equivalence of double brackets} is a warped morphism of double brackets with $\psi=\id_{A}$ such that $\Psi$ is an automorphism of $A^{\otimes 2}$.  
\end{definition}

The different ways to compare double brackets are depicted in Figure \ref{fig:Equiv}.  

\begin{lemma} \label{Lem:Ind-Equiv}
Let $\dgal{-,-}_1$ be a double bracket associated with a swap-commuting $A$-bimodule structure $\cdot_1$ on $A^{\otimes 2}$. 
Let $\Psi$ be an arbitrary automorphism of $A^{\otimes 2}$ commuting with the swap automorphism ${}^\circ$ of $A^{\otimes 2}$. Then 
\begin{equation} \label{Eq:Lem-Equiv1a}
    \cdot_2: A \times A^{\otimes 2} \times A \to A^{\otimes 2}, \quad 
    a \cdot_2 d \cdot_2 b := \Psi(a\cdot_1 \Psi^{-1}(d) \cdot_1 b)\,,
\end{equation}
defines a swap-commuting $A$-bimodule structure on $A^{\otimes 2}$ while 
\begin{equation} \label{Eq:Lem-Equiv2}
    \dgal{-,-}_2 = \Psi \circ \dgal{-,-}_1 \,,
\end{equation}
is a double bracket associated with $\cdot_2$. 
Furthermore, $\Psi$ induces an equivalence of the double brackets  $\dgal{-,-}_1$ and $\dgal{-,-}_2$.
\end{lemma}
\begin{proof}
Note that the assumption $\Psi(d^\circ)=(\Psi(d))^\circ$ for all $d\in A^{\otimes 2}$ also holds with $\Psi^{-1}$ in place of $\Psi$. This allows us to show that the swap $A$-bimodule $\ast_2$ associated with $\cdot_2$ \eqref{Eq:Lem-Equiv1a} satisfies 
\begin{equation} \label{Eq:Lem-Equiv1b}
\ast_2: A \times A^{\otimes 2} \times A \to A^{\otimes 2}, \quad 
    a \ast_2 d \ast_2 b := \Psi(a\ast_1 \Psi^{-1}(d) \ast_1 b)\,,     
\end{equation}
where $\ast_1$ is the swap $A$-bimodule of $\cdot_1$. 
In particular, we can check that \eqref{Eq:M-swapCom} for $\cdot_2,\ast_2$ is satisfied (i.e. $\cdot_2$ is swap-commuting) because $\cdot_1$ is swap-commuting. 

Since $\dgal{-,-}_1$ is cyclically antisymmetric and $\Psi$ commutes with ${}^\circ$, then $\dgal{-,-}_2$ satisfies \eqref{Eq:Dbr-cyc}. 
To check the right Leibniz rule \eqref{Eq:Dbr-Der1} for $\dgal{-,-}_2$ (with $\cdot_2$) using that of $\dgal{-,-}_1$ (with $\cdot_1$), we note  
\begin{align*}
    \dgal{a,bc}_2&= \Psi(b \cdot_1 \dgal{a,c}_1 + \dgal{a,b}_1 \cdot_1 c)
    = b \cdot_2 \Psi(\dgal{a,c}_1) + \Psi(\dgal{a,b}_1) \cdot_2 c \\
    &=b \cdot_2 \dgal{a,c}_2 + \dgal{a,b}_2 \cdot_2 c\,.
\end{align*}
In the same way, the left Leibniz rule \eqref{Eq:Dbr-Der2} is satisfied for $\dgal{-,-}_2$ (with $\ast_2$) because it holds for $\dgal{-,-}_1$ (with $\ast_1$). 
Hence, we have that $\dgal{-,-}_2$ is a double bracket. The fact that $\Psi$ defines an equivalence is direct from the definition of the second double bracket \eqref{Eq:Lem-Equiv2}.
\end{proof}

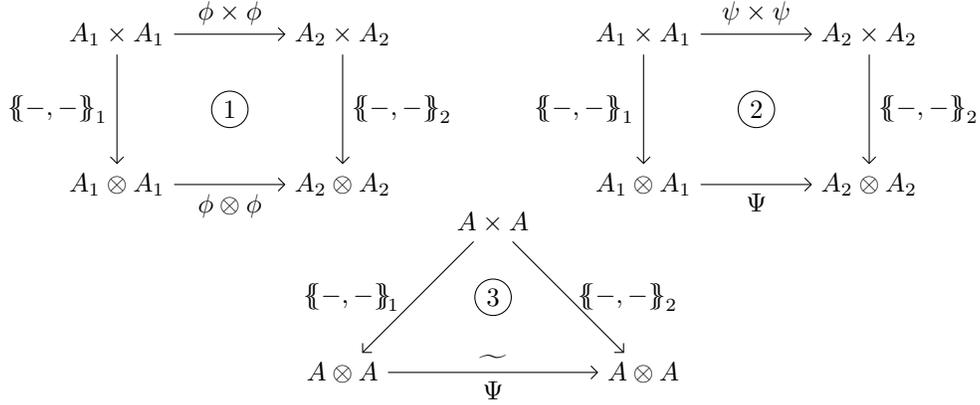
\begin{figure}
    \centering
\begin{tikzpicture}
 \node  (1TL) at (-8,1) {$A_1\times A_1$};
 \node  (1TR) at (-5,1) {$A_2\times A_2$};
  \node  (1BL) at (-8,-1) {$A_1\otimes A_1$};
 \node  (1BR) at (-5,-1) {$A_2\otimes A_2$};
 \node[circle,draw,inner sep=2pt] (1m) at (-6.5,0) {1};
\path[->,>=angle 90] (1TL) edge node[above] {$\phi\times \phi$}  (1TR) ;
\path[->,>=angle 90] (1TL) edge node[left] {$\dgal{-,-}_1$}  (1BL) ;
\path[->,>=angle 90] (1TR) edge node[right] {$\dgal{-,-}_2$}  (1BR) ;
\path[->,>=angle 90] (1BL) edge node[below] {$\phi\otimes \phi$}  (1BR) ;
 \node  (2TL) at (-1,1) {$A_1\times A_1$};
 \node  (2TR) at (2,1) {$A_2\times A_2$};
  \node  (2BL) at (-1,-1) {$A_1\otimes A_1$};
 \node  (2BR) at (2,-1) {$A_2\otimes A_2$};
 \node[circle,draw,inner sep=2pt] (2m) at (0.5,0) {2};
\path[->,>=angle 90] (2TL) edge node[above] {$\psi\times \psi$}  (2TR) ;
\path[->,>=angle 90] (2TL) edge node[left] {$\dgal{-,-}_1$}  (2BL) ;
\path[->,>=angle 90] (2TR) edge node[right] {$\dgal{-,-}_2$}  (2BR) ;
\path[->,>=angle 90] (2BL) edge node[below] {$\Psi$}  (2BR) ;
 \node  (3T) at (-3,-1.5) {$A\times A$};
  \node  (3BL) at (-5,-3.5) {$A\otimes A$};
 \node  (3BR) at (-1,-3.5) {$A\otimes A$};
 \node[circle,draw,inner sep=2pt] (3m) at (-3,-2.5) {3};
\path[->,>=angle 90] (3T) edge node[left] {$\dgal{-,-}_{\!1}\,\,$}  (3BL) ;
\path[->,>=angle 90] (3T) edge node[right] {$\dgal{-,-}_2$}  (3BR) ;
\path[->,>=angle 90] (3BL) edge node[below] {$\Psi$} (3BR);
\node (til) at  (-3,-3.4)  {$\widetilde{\hspace{0.5cm}}$}; 
\end{tikzpicture}
\caption{Commutative diagrams representing the definitions of: 1. Morphism of double brackets; 2. Warped morphism of double brackets; 3. Equivalence of double brackets.}
    \label{fig:Equiv}
\end{figure}

Finally, an equivalence of double brackets may or may not preserve double Poisson brackets. Two opposite behaviours are given in Example \ref{Ex:ctr-Equiv} and Proposition \ref{Pr:Op-Equiv}.

\begin{example} \label{Ex:ctr-Equiv}
Let $A$ be an arbitrary algebra, and assume that $\dgal{-,-}_1$ is a double bracket associated with $\cdot_1 = \cdot_{out}$ the outer bimodule structure \eqref{Eq:Mod-out} on $A^{\otimes 2}$. Fix an automorphism $\alpha$ of $A$. 
Using Lemma \ref{Lem:Ind-Equiv} with $\Psi=\alpha\otimes\alpha$, 
the map $ \dgal{-,-}_2 = (\alpha \otimes \alpha) \circ \dgal{-,-}_1$ 
is a double bracket associated with the $\alpha$-twisted outer bimodule structure $\cdot_2=\cdot_{out}^{\alpha}$ on $A^{\otimes 2}$, which satisfies 
$a\cdot_2 d \cdot_2 b = \alpha(a) d' \otimes d'' \alpha(b)$. 
In particular, $\alpha \otimes \alpha$ induces an equivalence of double brackets.

As a special case, consider $A= \mathbb{C}\langle x,y \rangle$ with 
\begin{equation}
\dgal{x,x}_1= x \otimes 1 - 1 \otimes x, \quad 
\dgal{y,y}_1= y \otimes 1 - 1 \otimes y, \quad 
\dgal{x,y}_1=0.
\end{equation}
(This completely determines a double bracket $\dgal{-,-}_1$ associated with the outer bimodule structure by Lemma \ref{Lem:Gen}.) 
From the work of Van den Bergh \cite{VdB1}, we have that $(A,\dgal{-,-}_1)$ is Poisson. Take $\alpha$ to be the automorphism $\alpha: x \mapsto y, y \mapsto x$. This yields an equivalence of double brackets between $\dgal{-,-}_1$ and $\dgal{-,-}_2$ (associated with the $\alpha$-twisted outer bimodule structure) defined by 
\begin{equation}
\dgal{x,x}_2= y \otimes 1 - 1 \otimes y, \quad 
\dgal{y,y}_2= x \otimes 1 - 1 \otimes x, \quad 
\dgal{x,y}_2=0.
\end{equation} 
This equivalent double bracket is not Poisson. 
Indeed, we claim that $\dgal{x,y,y}_2 \neq 0$, which implies that the double Jacobiator $\dgal{-,-,-}_2$ is nonzero. 
To see this, note that  
$\dgal{x,y,y}_2 = \dgal{x,\dgal{y,y}_2}_{2,L}$ by \eqref{Eq:dJac} because  $\dgal{x,y}_2=0$. We compute the above term,
\begin{align*}
    \dgal{x,\dgal{y,y}_2}_{2,L} 
&= \dgal{x, x\otimes 1}_{2,L} - \dgal{x,1 \otimes x}_{2,L} 
= y\otimes 1 \otimes 1 - 1 \otimes y \otimes 1 \,,
\end{align*}
which is nonzero and proves our claim. 
\end{example}

\begin{proposition} \label{Pr:Op-Equiv}
Let $\dgal{-,-}_1$ be a double Poisson bracket associated with a swap-commuting $A$-bimodule structure $\cdot_1$ on $A^{\otimes 2}$. 
Then $\dgal{-,-}_2 = \dgal{-,-}_1^{\circ}$ 
is a double Poisson bracket associated with the swap $A$-bimodule $\cdot_2=\ast_1$ of $\cdot_1$ on $A^{\otimes 2}$.  
Furthermore, the swap automorphism ${}^\circ$ induces an equivalence of the double brackets  $\dgal{-,-}_1$ and $\dgal{-,-}_2$.
\end{proposition}
\begin{proof}
Taking $\Psi={}^\circ$ in Lemma \ref{Lem:Ind-Equiv}, we have all the claims except the fact that $\dgal{-,-}_2$ is Poisson. For $a,b,c\in A$, we compute using the cyclic antisymmetry of $\dgal{-,-}_1$ that 
\begin{equation*}
    \dgal{a,\dgal{b,c}_2}_{2,L}
=- \dgal{a,\dgal{c,b}_1}_{2,L}
=-\tau_{(12)} \dgal{a,\dgal{c,b}_1}_{1,L}\,.
\end{equation*}
Thus, the double Jacobiator \eqref{Eq:dJac} becomes 
\begin{equation}
    \begin{aligned}
\dgal{a,b,c}_2=&
-\tau_{(12)}\dgal{a, \dgal{c,b}_1}_{1,L}
-\tau_{(123)}\tau_{(12)} \dgal{b, \dgal{a,c}_1}_{1,L}
-\tau_{(132)}\tau_{(12)} \dgal{c, \dgal{b,a}_1}_{1,L}\\
=&-\tau_{(12)}\left[ 
\dgal{a, \dgal{c,b}_1}_{1,L}
+\tau_{(123)} \dgal{c, \dgal{b,a}_1}_{1,L}
+\tau_{(132)} \dgal{b, \dgal{a,c}_1}_{1,L}
\right] \\
=&-\tau_{(12)} \dgal{a,c,b}_1\,. \label{Eq:Op-Equiv}    
    \end{aligned}
\end{equation}
Since $\dgal{-,-}_1$ is Poisson, we have $\dgal{-,-,-}_1=0$ which implies that  $\dgal{-,-,-}_2=0$ and thus $\dgal{-,-}_2$ is a double Poisson bracket. 
\end{proof}

We can adapt Proposition \ref{Pr:Op-Equiv} for double weak Poisson brackets. The following result concerns the different cases that are not obtained by permutation from one another, see the end of \ref{ss:wDBR}.

\begin{proposition} \label{Pr:Op-Equ-wk1}
Let $\dgal{-,-}_1$ be a double bracket associated with a swap-commuting $A$-bimodule structure $\cdot_1$ on $A^{\otimes 2}$. 
Let $\dgal{-,-}_2 = \dgal{-,-}_1^{\circ}$ be the equivalent (under the swap automorphism ${}^\circ$) double bracket associated with the swap $A$-bimodule $\cdot_2=\ast_1$ of $\cdot_1$ on $A^{\otimes 2}$.  
Then 

(i) $\dgal{-,-}_1$ is $(12)$-weak Poisson if and only if 
$\dgal{-,-}_2$ is $[(12),(13)]$-weak Poisson.

(ii) $\dgal{-,-}_1$ is $[(12),(23)]$-weak Poisson if and only if 
$\dgal{-,-}_2$ is $[(12),(23)]$-weak Poisson.
\end{proposition}
\begin{proof}
Using \eqref{Eq:Op-Equiv}, the double Jacobiators are related by 
$\dgal{-,-,-}_2=-\tau_{(12)}\circ \dgal{-,-,-}_1\circ \tau_{(23)}$. 

(i) For $\sigma=(12)$ and $\sigma'=(13)$, we can write that 
\begin{align*}
[\sigma,\sigma']\wk\!\dgal{-,-,-}_2
=&\dgal{-,-,-}_2-\tau_{(12)}\circ \dgal{-,-,-}_2\circ \tau_{(13)} \\
=&-\tau_{(12)}\circ \dgal{-,-,-}_1\circ \tau_{(23)}
+ \dgal{-,-,-}_1\circ \tau_{(123)}\\
=&-\tau_{(12)}\circ \left(\dgal{-,-,-}_1
-\tau_{(12)}\circ \dgal{-,-,-}_1\circ \tau_{(12)}\right) \circ \tau_{(23)}\\
=&-\tau_{(12)}\circ \sigma\wk\!\dgal{-,-,-}_1 \circ \tau_{(23)}\,.
\end{align*}
So the $\sigma$-weak Jacobiator of $\dgal{-,-}_1$ vanishes if and only if the $[\sigma,\sigma']$-weak Jacobiator of $\dgal{-,-}_2$ vanishes. 

(ii) For $\sigma=(12)$ and $\sigma'=(23)$, we have 
\begin{align*}
[\sigma,\sigma']\wk\!\dgal{-,-,-}_2
=&\dgal{-,-,-}_2-\tau_{(12)}\circ \dgal{-,-,-}_2\circ \tau_{(23)} \\
=&-\tau_{(12)}\circ \dgal{-,-,-}_1\circ \tau_{(23)}
+\dgal{-,-,-}_1
=[\sigma,\sigma']\wk\!\dgal{-,-,-}_1\,.
\end{align*}
We can directly conclude the result. 
\end{proof}

\subsection{Double brackets over a general base} \label{ss:base}

Let $B$ be a unital subalgebra of $A$. Given a double bracket $\dgal{-,-}$ on $A$, it is sometimes convenient to require the following \emph{$B$-linearity} condition: for any $b\in B$, $\dgal{b,-}:A\to A^{\otimes 2}$ (or equivalently   $\dgal{-,b}:A\to A^{\otimes 2}$) is identically zero. 
As was originally observed by Van den Bergh \cite{VdB1}, it is particularly interesting to consider $B$-linear double brackets when $A$ admits a complete set of orthogonal idempotents $(e_s)_{s\in I}$ and $B=\oplus_{s\in I} \kk e_s$. 
All the previous definitions and results can be adapted to the $B$-linear case in a straightforward way.


\section{Specific properties and examples ...} \label{S:Exmp}

We now describe interesting families of double brackets associated with a swap-commuting bimodule structure. To compare these families with Van den Bergh's original definition of double bracket \cite{VdB1}, we derive some of the properties that they share. 

\subsection{... from the outer bimodule structure} \label{ss:E-Out}

The case of double brackets associated with the outer bimodule structure \eqref{Eq:Mod-out} is due to Van den Bergh \cite{VdB1}. 
A more general consideration consists in endowing $A^{\otimes 2}$ with the $(\alpha,\beta)$-twisted outer bimodule structure (obtained by combining \eqref{Eq:Mod-Twis} and \eqref{Eq:Mod-out}). A double bracket on $A$ associated with this bimodule structure satisfies the right Leibniz rule 
\begin{equation} \label{Eq:DbrOut-Der1}
    \dgal{a,bc}=\alpha(b) \dgal{a,c}'\otimes \dgal{a,c}'' 
+ \dgal{a,b}' \otimes \dgal{a,b}'' \beta(c)\,, 
\end{equation}
and, using e.g. the cyclic antisymmetry \eqref{Eq:Dbr-cyc}, we have the left Leibniz rule 
\begin{equation} \label{Eq:DbrOut-Der2}
    \dgal{ac,b}=   \dgal{c,b}'\otimes \alpha(a) \dgal{c,b}'' 
+ \dgal{a,b}' \beta(c)\otimes \dgal{a,b}'' \,. 
\end{equation}
Thus, the left Leibniz rule is a derivation for the $(\alpha,\beta)$-twisted \emph{inner} bimodule structure. 

Next, we want a criterion to determine when a double bracket is Poisson. (This is trivially the case for the zero double bracket $\dgal{-,-}$.) 
Recall from Lemma \ref{Lem:Jac} that it suffices to verify the vanishing of the double Jacobiator on a set of generators, provided that it is a derivation in its third argument. 
By mimicking the proof of Proposition 2.3.1 in \cite{VdB1}, we get that the double Jacobiator \eqref{Eq:dJac} satisfies for any $a,b,c_1,c_2\in A$
\begin{equation}
\begin{aligned} \label{Eq:Out-dJac-all}
\dgal{a,b,c_1c_2}=&
\,\,(\alpha^2(c_1)\otimes 1 \otimes 1) \dgal{a,\dgal{b,c_2}}_L
+ \dgal{a,\dgal{b,c_1}}_L ( 1 \otimes 1\otimes \beta(c_2)) \\
&+(\alpha(c_1)\otimes 1 \otimes 1)\, \tau_{(123)} \dgal{b,\dgal{c_2,a}}_L
+\big(\tau_{(123)} \dgal{b,\dgal{c_1,a}}_L\big)\, (1 \otimes 1 \otimes \beta^2(c_2))\\
&+(\alpha(c_1)\otimes 1 \otimes 1)\, \tau_{(132)} \dgal{c_2,\dgal{a,b}}_L
+\big(\tau_{(132)} \dgal{c_1,\dgal{a,b}}_L\big)\, ( 1 \otimes 1 \otimes \beta(c_2))\\
&+\dgal{a,\alpha(c_1)}'\otimes \dgal{a,\alpha(c_1)}''  \beta(\dgal{b,c_2}') \otimes \dgal{b,c_2}'' \\
&+ \dgal{c_1,a}''\otimes \alpha(\dgal{c_1,a}') \dgal{b,\beta(c_2)}' \otimes \dgal{b,\beta(c_2)}''\,.
\end{aligned}
\end{equation}

\begin{lemma}[\cite{VdB1}] \label{Lem:OutJac}
In addition to the trivial case when $\dgal{-,-}=0$, 
the double Jacobiator \eqref{Eq:dJac} of a double bracket $\dgal{-,-}$ associated with the $(\alpha,\beta)$-twisted outer bimodule structure is a derivation in each of its entries when  $\alpha=\beta=\id_A$. 
In that case the third entry acts as the following derivation $A\to A^{\otimes 3}$: 
$\dgal{a,b,c_1c_2}=(c_1\otimes 1 \otimes 1) \dgal{a,b,c_2} 
+ \dgal{a,b,c_1} (1\otimes 1 \otimes c_2)$, 
for all $a,b,c_1,c_2\in A$.  
\end{lemma} 
\begin{proof}
The case $\alpha=\beta=\id_A$ for the third entry holds because the last two terms in \eqref{Eq:Out-dJac-all} cancel out by \eqref{Eq:Dbr-cyc}. 
We get derivations in the other two entries by Lemma \ref{Lem:Jac}. 
\end{proof}

Note that Lemma \ref{Lem:OutJac} does not exhaust all possible cases where $\dgal{-,-,-}$ is a derivation in its arguments, as the following example shows.

\begin{example}
Take $A=\kk\langle x,y \rangle/(x^2,y^2,xy,yx)$. We can extend 
\begin{equation}
    \dgal{x,x}=x\otimes y - y \otimes x\,, \quad \dgal{x,y}=\dgal{y,x}=0\,,
\end{equation}
to a double bracket associated with the $\alpha$-twisted outer bimodule, for the automorphism $\alpha(x)= y$, $\alpha(y)= x$. Indeed, it suffices to check that $\dgal{-,-}$ vanishes on any product of generators $x,y$ after applying Leibniz rules \eqref{Eq:DbrOut-Der1} and \eqref{Eq:DbrOut-Der2} (with $\beta=\alpha$). 
Next, we note that the double Jacobiator of this double bracket vanishes on any triple of generators $x,y$, see e.g. \cite[\S2.6]{DSKV} (this does not require any Leibniz rule). Since any product of at least two generators is zero in $A$, the double bracket is Poisson.  
\end{example}

Given an arbitrary double bracket $\dgal{-,-}$ associated with the $(\alpha,\beta)$-twisted outer bimodule structure, we can consider the $\kk$-bilinear operation  
\begin{equation} \label{Eq:Out-brm}
    \br{-,-}_\mult = \mult \circ \dgal{-,-} : A\times A\to A\,, \quad 
    \br{a,b}_\mult = \dgal{a,b}' \dgal{a,b}''\,,
\end{equation}
where $\mult$ is the multiplication on $A$ seen as a map $A\otimes A \to A$. We get the next result from \eqref{Eq:DbrOut-Der1}.

\begin{lemma} \label{Lem:Out-Ind-der}
The map $\br{-,-}_\mult$ is an $(\alpha,\beta)$-twisted derivation in its second argument. That is, for any $a,b,c\in A$, 
$\br{a,bc}_\mult=\alpha(b) \br{a,c}_\mult + \br{a,b}_\mult \beta(c)$. 
\end{lemma}

\begin{lemma} \label{Lem:Out-Ind-2}
If $\alpha=\beta$, then $\br{-,-}_\mult: A\times A\to A$ descends to a map $\br{-,-}_\mult:A/[A,A]\times A\to A$, which is an $\alpha$-twisted derivation in its second argument. 
Furthermore, this descends to an antisymmetric map 
$\br{-,-}_\mult:A/[A,A]\times A/[A,A]\to A/[A,A]$.
\end{lemma}
\begin{proof}
For the first part, note that the identity 
\begin{equation*}
    \br{ab,c}_\mult=\dgal{a,c}'\alpha(b)\dgal{a,c}''
    + \dgal{b,c}' \alpha(a) \dgal{b,c}''\,,
\end{equation*}
implies that $\br{ab-ba,c}_\mult=0$ for any $a,b,c\in A$. 
In particular, the derivation property follows from Lemma \ref{Lem:Out-Ind-der}. 
For the second part, we compute for $\br{-,-}_\mult:A/[A,A]\times A\to A$ that  
\begin{equation*}
\br{a,bc-cb}_\mult 
=[\alpha(b),\br{a,c}_\mult] + [\br{b,c}_\mult , \alpha(a)]\,,
\end{equation*}
so that we have a well-defined map $A/[A,A]\times A/[A,A]\to A/[A,A]$. Its antisymmetry follows from \eqref{Eq:Dbr-cyc}: 
$\br{a,b}_\mult=-\dgal{b,a}'' \dgal{b,a}'$, which is $-\br{b,a}_\mult$ modulo commutators. 
\end{proof}

In the case $\alpha=\beta=\id_A$, both lemmas can be found in \cite[\S2.4]{VdB1}. Furthermore, it is proven that in the presence of a double Poisson bracket, the map $\br{-,-}_\mult$ \eqref{Eq:Out-brm} is a left Loday bracket\footnote{A left Loday bracket is a bilinear map $[-,-]$ on a vector space $V$ satisfying, for any $a,b,c\in V$,  $[a,[b,c]]=[[a,b],c]+[b,[a,c]]$. 
It is a right Loday bracket if instead $[[a,b],c]=[a,[b,c]]+[[a,c],b]$. 
An antisymmetric Loday bracket is therefore a Lie bracket. 
We follow the terminology of \cite{VdB1}, as Loday brackets are usually called Leibniz brackets \cite{Lo}.} which descends to a Lie bracket on $A/[A,A]$. 
In the $(\alpha,\beta)$-twisted case, we can compute that 
\begin{equation}
\begin{aligned} \label{Eq:Out-Loday}
  & \br{a,\br{b,c}_\mult}_\mult  - \br{b,\br{a,c}_\mult}_\mult 
   - \br{\br{a,b}_\mult,c}_\mult \\
=&\,\, \mult \circ \beta_3 \dgal{a,\dgal{b,c}}_L
+ \mult \circ \alpha_1 \circ \tau_{(123)} \dgal{b,\dgal{c,a}}_L
+ \mult \circ \beta_2 \circ \tau_{(132)} \dgal{c,\dgal{a,b}}_L \\
&- \mult \circ \beta_3 \dgal{b,\dgal{a,c}}_L
- \mult \circ \alpha_1 \circ \tau_{(123)} \dgal{a,\dgal{c,b}}_L
- \mult \circ \alpha_2 \circ \tau_{(132)} \dgal{c,\dgal{b,a}}_L\,. 
\end{aligned}
\end{equation}
Here we extended the multiplication by $\mult:A^{\otimes 3}\to A$, while for each $1\leq i \leq 3$ an automorphism $\alpha$ of $A$ defines an automorphism $\alpha_i$ of $A^{\otimes 3}$ acting by $\alpha$ on the $i$-th factor and by $\id_A$ on the other two, e.g. $\alpha_1=\alpha \otimes \id_A\otimes \id_A$. 
Except in particular cases (e.g. $\alpha=\beta=\id_A$), the right-hand side of \eqref{Eq:Out-Loday} is not a difference of two double Jacobiators. Hence  \eqref{Eq:Out-Loday} does not vanish in general even for a double Poisson bracket, and $\br{-,-}_\mult$ is not a left Loday bracket. 
Nevertheless, we can modify Van den Bergh's result as follows. 

\begin{proposition} \label{Pr:Out-Lod}
Assume that $\dgal{-,-}$ is a double bracket associated with the $\alpha$-twisted outer bimodule structure. Introduce the (left) \emph{$\alpha$-twisted double Jacobiator} for any $a,b,c\in A$ by 
\begin{equation}  \label{Eq:Out-LodAlph}
\dgal{a,b,c}_{\alpha}:=   
 \alpha_3 \dgal{a,\dgal{b,c}}_L
+ \tau_{(123)} \circ \alpha_3 \dgal{b,\dgal{c,a}}_L
+  \tau_{(132)} \circ \alpha_3 \dgal{c,\dgal{a,b}}_L\,.
\end{equation}
If $\mult \circ \dgal{-,-,-}_{\alpha} - \mult \circ \dgal{-,-,-}_{\alpha}\circ \tau_{(12)}=0$, then $\br{-,-}_\mult$ is a left Loday bracket on $A$. In particular, it descends to a Lie bracket on $A/[A,A]$. 
\end{proposition}
\begin{proof}
To have a left Loday bracket, we need that the right-hand side of \eqref{Eq:Out-Loday} vanishes. But it equals $\mult \circ \dgal{a,b,c}_{\alpha} - \mult \circ \dgal{b,a,c}_{\alpha}$, which is zero by assumption. 
The map induced on $A/[A,A]$ is antisymmetric by Lemma \ref{Lem:Out-Ind-2}, which finishes the proof since an antisymmetric Loday bracket is a Lie bracket.
\end{proof}
The operation  $\dgal{-,-,-}_{\alpha} -  \dgal{-,-,-}_{\alpha}\circ \tau_{(12)}:A^{\times 3}\to A^{\otimes3}$
does not enjoy nice derivation rules (compare with \eqref{Eq:Out-dJac-all}), so checking the assumption of Proposition \ref{Pr:Out-Lod} seems to be a challenging task for $\alpha\neq \id_A$. 
When $\alpha=\id_A$ the map \eqref{Eq:Out-LodAlph} is just the double Jacobiator \eqref{Eq:dJac}, so Proposition \ref{Pr:Out-Lod} specialises to Van den Bergh's result \cite[\S2.4]{VdB1}. 

\medskip

Let us finally comment on the equivalences between double brackets associated with twisted outer bimodule structures.
First, by reverting the argument made in Example \ref{Ex:ctr-Equiv}, we can check that any double bracket associated with the $\alpha$-twisted outer bimodule structure (where $\beta=\alpha$) is equivalent to a double bracket associated with the outer bimodule structure (where $\beta=\alpha=\id_A$). 
Second, this argument can be adapted in the $(\alpha,\beta)$-twisted case to give an equivalence with a double bracket associated either with the $(\gamma,\id_A)$-twisted or the $(\id_A,\gamma)$-twisted outer bimodule structure (with $\gamma=\beta^{-1}\circ \alpha$ or $\gamma=\alpha^{-1}\circ \beta$, respectively). 
Note, however, that such equivalences do \emph{not} always preserve the Poisson property as Example \ref{Ex:ctr-Equiv} shows. 
This is precisely the reason why the $\alpha$-twisted double Jacobiator \eqref{Eq:Out-LodAlph} appears in Proposition \ref{Pr:Out-Lod}.

\subsection{... from the inner bimodule structure} \label{ss:E-Int}

Given a double bracket $\dgal{-,-}$ on $A$ which is associated with the $(\alpha,\beta)$-twisted inner bimodule structure, it satisfies the Leibniz rules  
\begin{align} \label{Eq:DbrIn-Der1}
    \dgal{a,bc}= \dgal{a,c}'\otimes \alpha(b)\dgal{a,c}'' 
+ \dgal{a,b}'\beta(c) \otimes \dgal{a,b}'' \,, \\
\label{Eq:DbrIn-Der2}
    \dgal{ac,b}=  \alpha(a) \dgal{c,b}'\otimes  \dgal{c,b}'' 
+ \dgal{a,b}' \otimes \dgal{a,b}'' \beta(c) \,. 
\end{align}
These Leibniz rules are consistent with the cyclic antisymmetry \eqref{Eq:Dbr-cyc}. 
By Lemma \ref{Lem:Ind-Equiv}, the swap automorphism ${}^\circ$ induces an equivalence with a double bracket associated with the $(\alpha,\beta)$-twisted \emph{outer} bimodule structure. Thus, many results covered in \ref{ss:E-Out} in the outer case can be directly rewritten, or proven in the same way. We start with the following reformulation of Proposition \ref{Pr:Op-Equiv}.

\begin{theorem} \label{Thm:Poi-OutIn}
A double bracket $\dgal{-,-}$ on $A$ associated with the $(\alpha,\beta)$-twisted inner bimodule structure is Poisson if and only if the equivalent double bracket $\dgal{-,-}^\circ$ associated with the $(\alpha,\beta)$-twisted outer bimodule structure is Poisson.
\end{theorem}

\begin{lemma} \label{Lem:InJac}
In addition to the trivial case when $\dgal{-,-}=0$, 
the double Jacobiator \eqref{Eq:dJac} of a double bracket $\dgal{-,-}$ associated with the $(\alpha,\beta)$-twisted inner bimodule structure is a derivation in each of its entries when  $\alpha=\beta=\id_A$. 
In that case the second entry acts as the following derivation $A\to A^{\otimes 3}$: 
$\dgal{a,b_1b_2,c}=(1\otimes b_1 \otimes 1) \dgal{a,b_2,c} 
+ \dgal{a,b_1,c} (1\otimes 1 \otimes b_2)$, 
for all $a,b_1,b_2,c\in A$.  
\end{lemma} 
\begin{proof}
It suffices to combine Lemma \ref{Lem:OutJac} with the fact that 
$\dgal{-,-,-}=-\tau_{(12)}\circ \dgal{-,-,-}_2 \circ \tau_{(23)}$ (see the proof of Proposition \ref{Pr:Op-Equiv}), where $\dgal{-,-,-}_2$ is the double Jacobiator of $\dgal{-,-}^\circ$.
\end{proof}

\begin{lemma} \label{Lem:In-Ind-der}
If $\dgal{-,-}$ is associated with the $(\alpha,\beta)$-twisted inner bimodule structure, the $\kk$-bilinear operation 
\begin{equation} \label{Eq:in-brm}
    \br{-,-}_\mult = \mult \circ \dgal{-,-} : A\times A\to A\,, \quad 
    \br{a,b}_\mult = \dgal{a,b}' \dgal{a,b}''\,,
\end{equation}
is an $(\alpha,\beta)$-twisted derivation in its first entry. I.e.  
$\br{ab,c}_\mult=\alpha(a) \br{b,c}_\mult + \br{a,c}_\mult \beta(b)$, $\forall a,b,c\in A$. 
When $\alpha=\beta$, the map $\br{-,-}_\mult$ \eqref{Eq:in-brm} descends to a map $\br{-,-}_\mult:A\times A/[A,A]\to A$, which is an $\alpha$-twisted derivation in its first argument. 
Furthermore, this descends to an antisymmetric map 
$\br{-,-}_\mult:A/[A,A]\times A/[A,A]\to A/[A,A]$.
\end{lemma}
\begin{proof}
Easily adapted from Lemmas \ref{Lem:Out-Ind-der} and \ref{Lem:Out-Ind-2}. 
\end{proof}

\begin{proposition} \label{Pr:In-Lod}
Assume that $\dgal{-,-}$ is a double bracket associated with the $\alpha$-twisted inner bimodule structure. Introduce the (right) \emph{$\alpha$-twisted double Jacobiator} for any $a,b,c\in A$ by 
\begin{equation} \label{Eq:In-LodAlph}
\dgal{a,b,c}_{R,\alpha}:=   
 \alpha_1 \dgal{\dgal{a,b},c}_R
+ \tau_{(123)} \circ \alpha_1 \dgal{\dgal{b,c},a}_R
+  \tau_{(132)} \circ \alpha_1 \dgal{\dgal{c,a},b}_R\,.
\end{equation}
If $\mult \circ \dgal{-,-,-}_{R,\alpha} - \mult \circ \dgal{-,-,-}_{R,\alpha}\circ \tau_{(23)}=0$, then $\br{-,-}_\mult$ is a right Loday bracket on $A$. In particular, it descends to a Lie bracket on $A/[A,A]$. 
\end{proposition}
\begin{proof}
The definition of \eqref{Eq:In-LodAlph} uses \eqref{Eq:RightDmap}. It suffices to show that we have a right Loday bracket on $A$ under the given assumptions, i.e. we need the vanishing of 
\begin{equation} \label{Eq:In-Lod-pf1}
\br{\br{a,b}_{\mult},c}_{\mult}
-\br{\br{a,c}_{\mult},b}_{\mult}
-\br{a,\br{b,c}_{\mult}}_{\mult}\,.    
\end{equation}
We can compute, thanks to the cyclic antisymmetry \eqref{Eq:Dbr-cyc} of $\dgal{-,-}$, that 
\begin{align*}
    \br{\br{a,b}_{\mult},c}_{\mult}&=\br{\dgal{a,b}'\dgal{a,b}'',c}_{\mult} \\
&=\mult \,\left[
\alpha(\dgal{a,b}')\otimes \dgal{\dgal{a,b}'',c}
+\dgal{\dgal{a,b}',c}\otimes \alpha(\dgal{a,b}'')
\right] \\
&= \mult \circ \alpha_1 \, \dgal{\dgal{a,b},c}_R 
- \mult \circ \tau_{(132)}\circ \alpha_1\, \dgal{\dgal{b,a},c}_R\,,
\end{align*}
and similarly 
\begin{align*}
\br{\br{a,c}_{\mult},b}_{\mult}&=   
\mult \circ \alpha_1 \, \dgal{\dgal{a,c},b}_R 
- \mult \circ \tau_{(132)}\circ \alpha_1\, \dgal{\dgal{c,a},b}_R\,, \\
\br{a,\br{b,c}_{\mult}}_{\mult}&=
\mult \circ \tau_{(123)}\circ \alpha_1 \, \dgal{\dgal{c,b},a}_R 
- \mult \circ \tau_{(123)}\circ \alpha_1\, \dgal{\dgal{b,c},a}_R\,.
\end{align*}
Gathering these identities, we see that \eqref{Eq:In-Lod-pf1} equals  
\begin{align*}
&\mult \, \left[\alpha_1 \dgal{\dgal{a,b},c}_R
+ \tau_{(123)} \circ \alpha_1 \dgal{\dgal{b,c},a}_R
+  \tau_{(132)} \circ \alpha_1 \dgal{\dgal{c,a},b}_R  \right] \\
-&\mult \, \left[\alpha_1 \dgal{\dgal{a,c},b}_R
+ \tau_{(123)} \circ \alpha_1 \dgal{\dgal{c,b},a}_R
+  \tau_{(132)} \circ \alpha_1 \dgal{\dgal{b,a},c}_R  \right]\,,
\end{align*}
which vanishes under the given assumption. 
\end{proof}
We have the equality $\dgal{-,-,-}=\tau_{(12)}\circ \dgal{-,-,-}_{R,\id_A}\circ \tau_{(12)}$ between the double Jacobiator originally defined in \eqref{Eq:dJac} and its right twisted version \eqref{Eq:In-LodAlph}. 
This follows from a direct comparison of \eqref{Eq:In-LodAlph} and  \eqref{Eq:dJac-InR} in the case $\alpha=\id_A$. 
Thus, the assumptions of Proposition \ref{Pr:In-Lod} are satisfied for a double Poisson bracket associated with the inner bimodule structure (i.e. when $\alpha=\id_A$). 

\begin{theorem} \label{Thm-InOut}
Let $\dgal{-,-}_1$ and $\dgal{-,-}_2$ be double brackets associated with the inner and outer bimodule structures, respectively. 
Assume that they are equivalent under the swap automorphism: $\dgal{-,-}_1=\dgal{-,-}_2^\circ$. 
If one of them is Poisson, then the operations   
$\br{-,-}_{j,\mult}:= \mult \circ \dgal{-,-}_j$, $j=1,2$,  are Loday brackets (right for $j=1$, left for $j=2$) related through 
\begin{equation} \label{Eq:Thm-InOut}
\br{a,b}_{1,\mult} = - \br{b,a}_{2,\mult}\,.   
\end{equation}
In particular, they induce the same Lie bracket on $A/[A,A]$.
\end{theorem}
\begin{proof}
By Proposition \ref{Pr:Op-Equiv}, if one double bracket is Poisson, then both double brackets are. Therefore, by Propositions \ref{Pr:Out-Lod} (in the outer case) and \ref{Pr:In-Lod} (in the inner case), we get that the maps $\br{-,-}_{j,\mult}$ are Loday brackets which descend to Lie brackets. 
So we are left to prove \eqref{Eq:Thm-InOut}. This follows from   
$$\br{a,b}_{1,\mult} = \mult \circ \dgal{a,b}_1=- \mult \circ \dgal{b,a}_1^\circ = - \mult \circ \dgal{b,a}_2=-\br{b,a}_{2,\mult}\,,$$
using the equivalence of the two double brackets and the cyclic antisymmetry of $\dgal{-,-}_1$. 
\end{proof}

\begin{remark}
It is well-known \cite{Lo} that if $[-,-]$ is a left (resp. right) Loday bracket, then $[x,y]_{\bullet}=[y,x]$ defines a right (resp. left) Loday bracket. Thus, Theorem \ref{Thm-InOut} can be interpreted as lifting this principle to double brackets when $[-,-]=\br{-,-}_\mult$ is induced by a double Poisson bracket associated with the outer (resp. inner) bimodule structure.
\end{remark}

\subsection{... from the right bimodule structure} \label{ss:E-Right}

We endow $A^{\otimes 2}$ with the $(\alpha,\beta)$-twisted right bimodule structure (obtained by twisting \eqref{Eq:Mod-R}). A double bracket on $A$ associated with this bimodule structure satisfies the right Leibniz rule 
\begin{equation} \label{Eq:DbrR-Der1}
    \dgal{a,bc}= \dgal{a,c}'\otimes \alpha(b)\dgal{a,c}'' 
+ \dgal{a,b}'  \otimes \dgal{a,b}''\beta(c)\,, 
\end{equation}
while we have the left Leibniz rule 
\begin{equation} \label{Eq:DbrR-Der2}
    \dgal{ac,b}=  \alpha(a)\dgal{c,b}'\otimes\dgal{c,b}'' 
+ \dgal{a,b}'\beta(c) \otimes \dgal{a,b}'' \,, 
\end{equation}
meaning that the left Leibniz rule is a derivation for the $(\alpha,\beta)$-twisted \emph{left} bimodule structure. 
Let us remind the reader that the case $\alpha=\beta=\id_A$ should be seen as a formalisation of the operation $\br{-\stackrel{\otimes}{,}-}$ (see the introduction) which is well-known in the integrable systems community. In particular, many results presented below are expected. 

\medskip

A short computation yields that the double Jacobiator  \eqref{Eq:dJac} satisfies for any $a,b,c_1,c_2\in A$
\begin{align*}
\dgal{a,b,c_1c_2}=&
\,\,( 1 \otimes 1\otimes \alpha(c_1) ) \dgal{a,\dgal{b,c_2}}_L
+ \dgal{a,\dgal{b,c_1}}_L (1 \otimes 1\otimes \beta(c_2)) \\
&+( 1 \otimes 1\otimes \alpha^2(c_1) )\, \tau_{(123)} \dgal{b,\dgal{c_2,a}}_L
+\big(\tau_{(123)} \dgal{b,\dgal{c_1,a}}_L\big)\, (1 \otimes 1\otimes \beta^2(c_2))\\
&+( 1 \otimes 1\otimes \alpha(c_1) )\, \tau_{(132)} \dgal{c_2,\dgal{a,b}}_L
+\big(\tau_{(132)} \dgal{c_1,\dgal{a,b}}_L\big)\, (1 \otimes 1\otimes \beta(c_2))\\
&+\dgal{c_2,a}'' \otimes \dgal{b,\alpha(c_1)}' \otimes \dgal{b,\alpha(c_1)}'' \beta(\dgal{c_2,a}')\\
&+ \dgal{c_1,a}''\otimes \dgal{b,\beta(c_2)}' \otimes \alpha(\dgal{c_1,a}') \dgal{b,\beta(c_2)}''\,.
\end{align*}
Contrary to the study of double brackets associated with the outer/inner bimodule, in the case $\alpha=\beta=\id_A$ the double Jacobiator fails to be a derivation in all its arguments in general as we have 
\begin{equation}
\begin{aligned} \label{Eq:dJac-R-der}
&\dgal{a,b,c_1c_2}=
( 1 \otimes 1\otimes c_1 ) \dgal{a,b,c_2}
+ \dgal{a,b,c_1} (1 \otimes 1\otimes c_2) \\
&\qquad +\dgal{c_2,a}'' \otimes \dgal{b,c_1}' \otimes \dgal{b,c_1}'' \dgal{c_2,a}'
+ \dgal{c_1,a}''\otimes \dgal{b,c_2}' \otimes \dgal{c_1,a}' \dgal{b,c_2}''\,.
\end{aligned}
\end{equation}
We can nevertheless state the following result in the case $\alpha=\beta=\id_A$.

\begin{lemma} \label{Lem:R-wJac}
For $\sigma=(12)$, 
The $\sigma$-weak double Jacobiator \eqref{Eq:wdJac} of a double bracket $\dgal{-,-}$ associated with the right bimodule structure is a derivation in each of its entries. 
In that case the third entry acts as the following derivation $A\to A^{\otimes 3}$: 
$\sigma\wk\!\dgal{a,b,c_1c_2}=(1\otimes 1\otimes c_1)\,\, \sigma\wk\!\dgal{a,b,c_2} 
+ \sigma\wk\!\dgal{a,b,c_1} \,(1\otimes 1 \otimes c_2)$, 
for all $a,b,c_1,c_2\in A$.  
\end{lemma} 
\begin{proof}
Combining the definition \eqref{Eq:wdJac} with \eqref{Eq:dJac-R-der}, we have for $\sigma=(12)$ that 
\begin{align*}
   & \sigma\wk\!\dgal{a,b,c_1c_2}=\dgal{a,b,c_1c_2}-\tau_{(12)}\dgal{b,a,c_1c_2} \\
&\quad =( 1 \otimes 1\otimes c_1 ) \,\,\sigma\wk\!\dgal{a,b,c_2}
+ \sigma\wk\!\dgal{a,b,c_1}\,\, (1 \otimes 1\otimes c_2) \\
&\qquad +\dgal{c_2,a}'' \otimes \dgal{b,c_1}' \otimes \dgal{b,c_1}'' \dgal{c_2,a}'
+ \dgal{c_1,a}''\otimes \dgal{b,c_2}' \otimes \dgal{c_1,a}' \dgal{b,c_2}''\\
&\qquad 
-\dgal{a,c_1}' \otimes\dgal{c_2,b}'' \otimes  \dgal{a,c_1}'' \dgal{c_2,b}'
- \dgal{a,c_2}' \otimes\dgal{c_1,b}''\otimes  \dgal{c_1,b}' \dgal{a,c_2}''\,.
\end{align*}
Due to the cyclic antisymmetry \eqref{Eq:Dbr-cyc}, the third and sixth (resp. fourth and fifth) terms in the second equality cancel out. 
This gives the claimed derivation property for the $\sigma$-weak double Jacobiator in the third argument, hence we have a derivation in each entry by Lemma \ref{Lem:wJac}.  
\end{proof}

As a consequence of Lemma \ref{Lem:R-wJac}, double $(12)$-weak Poisson brackets are more interesting than double Poisson brackets in the case associated with the right bimodule structure. This is illustrated in the next example.

\begin{example}
Let $A=\kk\langle x,y\rangle$. 
By Lemma \ref{Lem:Gen}, for any $\lambda\in \kk^\times$, we have a double bracket satisfying 
$$\dgal{x,x}=0,\quad \dgal{y,y}=0,\quad \dgal{x,y}=\lambda\, 1\otimes 1\,.$$
Combining Lemmas \ref{Lem:wJac2} and  \ref{Lem:R-wJac}, this double bracket is $(12)$-weak Poisson. However, it is not Poisson: although the double Jacobiator vanishes on generators, we have 
$\dgal{x,x,y^2}=-2\lambda^2\, 1\otimes 1 \otimes 1$.
\end{example}

Another difference with the outer/inner cases is that double brackets associated with the (twisted) right bimodule structure do not enjoy nice properties after multiplication of the two components as in \eqref{Eq:Out-brm}. In the present situation, there is no need to multiply factors of the double bracket if we quotient out by  commutators.

\begin{lemma} \label{Lem:R-Ind-der}
If $\dgal{-,-}$ is associated with the $\alpha$-twisted right bimodule structure (when $\beta=\alpha$), the double bracket descends to maps  
\begin{equation} \label{Eq:in-brBullet}
    {}_\bullet\!\dgal{-,-} : A/[A,A]\times A\to A/[A,A]\otimes A\,, \quad 
    \dgal{-,-}_\bullet :A\times A/[A,A]  \to A \otimes  A/[A,A]\,.
\end{equation}
These operations satisfy the following  $\alpha$-twisted derivations:
\begin{align*}
    {}_\bullet\!\dgal{a,bc}&=b\cdot_r^\alpha {}_\bullet\!\dgal{a,c} + {}_\bullet\!\dgal{a,b}\cdot_r^\alpha c\,, \\
    \dgal{bc,a}_\bullet&=b\cdot_l^\alpha \dgal{c,a} + \dgal{b,a}\cdot_l^\alpha c\,,
\end{align*}
where $\cdot_r^\alpha$ (resp. $\cdot_l^\alpha$) is the bimodule structure on $A/[A,A]\otimes A$ (resp. $A\otimes A/[A,A]$) induced by the $\alpha$-twisted right (resp. left) bimodule structure on $A^{\otimes 2}$.  
Furthermore, both maps descend to the same operation  
${}_\bullet\!\dgal{-,-}_\bullet :A/[A,A]\times A/[A,A]\to A/[A,A]\otimes A/[A,A]$.
\end{lemma}
\begin{proof}
To see that $\dgal{-,-}_\bullet$ is well-defined, it suffices to note that by \eqref{Eq:DbrR-Der1} 
\begin{equation}
    \dgal{a,bc-cb}=\dgal{a,b}'\otimes [\dgal{a,b}'',\alpha(c)] + \dgal{a,c}'\otimes [\alpha(b),\dgal{a,c}'']\,,
\end{equation}
so that $\dgal{a,-}$ sends $[A,A]$ to $A\otimes [A,A]$ for any $a\in A$. Next, $\dgal{-,-}_\bullet$ inherits from the left Leibniz rule  \eqref{Eq:DbrR-Der2} the stated $\alpha$-twisted left derivation in its first argument. 
The case of ${}_\bullet\!\dgal{-,-}$ is similar. 
Gathering both constructions, we can prove that we have a well-defined map ${}_\bullet\!\dgal{-,-}_\bullet$. 
\end{proof}

\begin{proposition} \label{Pr:R-Ind}
Let $\dgal{-,-}$ be a double $(12)$-weak Poisson bracket associated with the right bimodule structure.  
Then the operation ${}_\bullet\!\dgal{-,-}_\bullet$ from Lemma \ref{Lem:R-Ind-der} uniquely extends  to a Poisson bracket on $\Sym(A/[A,A])$. 
\end{proposition}
\begin{proof}
To ease notations, we also denote by $a\in A$ the image of such an element in $A/[A,A]$. We write the commutative product on the symmetric tensor algebra $\Sym(A/[A,A])$ by $\bullet$. 
For any $n \geq1$, $\pi_n:A^{\otimes n}\to \Sym(A/[A,A])$ is the morphism defined by 
$\pi_n(a_1\otimes \ldots \otimes a_n)=a_1 \bullet \ldots \bullet a_n$. It will be useful to note that for any $\sigma\in S_n$, $\pi_n \circ \tau_\sigma = \pi_n$. 

We claim that the desired Poisson bracket is the unique operation satisfying  
$$\br{-,-}_{\Sym}:\Sym(A/[A,A])\times \Sym(A/[A,A])\to \Sym(A/[A,A])\,, 
\br{a,b}_{\Sym}=\pi_2\,\dgal{a,b}\,.$$
(It is unique by extension through the usual Leibniz rules of a Poisson bracket.)
This operation is given on elements of $A/[A,A]$, seen as generators, by extending the operation ${}_\bullet\!\dgal{-,-}_\bullet$ through $(A/[A,A])^{\times 2}\to \Sym(A/[A,A])$ by symmetric multiplication of the two components. 
Antisymmetry follows from the cyclic antisymmetry \eqref{Eq:Dbr-cyc}: 
$$\br{a,b}_{\Sym}=\pi_2\,\dgal{a,b}=
-\pi_2\, \dgal{b,a}^\circ=-\pi_2\,\dgal{b,a} 
= -\br{b,a}_{\Sym}\,.$$ 
Jacobi identity follows from the vanishing of the $(12)$-weak double Jacobiator. Indeed, we have that  
\begin{align*}
\br{a,\br{b,c}_{\Sym}}_{\Sym}
&=\br{a,\dgal{b,c}'}_{\Sym}\bullet \dgal{b,c}''
+ \dgal{b,c}'\bullet \br{a,\dgal{b,c}''}_{\Sym}\\
&=\pi_3(\dgal{a,\dgal{b,c}}_L + \dgal{a,\dgal{b,c}}_R)\,.
\end{align*}
We can then write  
\begin{align*}
&\br{a,\br{b,c}_{\Sym}}_{\Sym} + \br{b,\br{c,a}_{\Sym}}_{\Sym}
+ \br{c,\br{a,b}_{\Sym}}_{\Sym} \\
=&\pi_3(\dgal{a,\dgal{b,c}}_L+\tau_{(123)}\dgal{b,\dgal{c,a}}_L + \tau_{(132)} \dgal{c,\dgal{a,b}}_L) \\
&+\pi_3( \dgal{a,\dgal{b,c}}_R+\tau_{(123)}\dgal{c,\dgal{a,b}}_R +\tau_{(132)} \dgal{b,\dgal{c,a}}_R)\,.
\end{align*}
Using \eqref{Eq:dJac-R}, this equals 
$\pi_3(\dgal{a,b,c}-\dgal{b,a,c})=
\pi_3(\dgal{a,b,c}-\tau_{(12)}\dgal{b,a,c})$. 
By definition, this is the image under $\pi_3$ of the 
 $(12)$-weak double Jacobiator $\sigma\wk\!\dgal{a,b,c}$ with  $\sigma=(12)$. 
\end{proof}

Instead of working on $A/[A,A]$ and $\Sym(A/[A,A])$ as in the previous two results, we can consider the abelianization of $A$, given by $A^{\ab}=A/J_A$ where $J_A$ is the (two-sided) ideal generated by elements of $[A,A]$. 

\begin{lemma} \label{Lem:Dbr-Ab-R}
If $\dgal{-,-}$ is associated with the $\alpha$-twisted right bimodule structure, then the double bracket descends to $A^{\ab}$.
\end{lemma}
\begin{proof}
 We get that for any $a,b,c\in A$ and $j\in [A,A]$, 
\begin{align*}
\dgal{a,b j c}=
\dgal{a,b}'\otimes \dgal{a,b}'' \alpha(j) \alpha(c) 
+ \alpha(b)\dgal{a,j'}\otimes \dgal{a,j}'' \alpha(c) 
+\alpha(b)\alpha(j) \dgal{a,c}'\otimes \dgal{a,c}''\,.
\end{align*}
Since $\alpha(j)\in [A,A]$, the first and third terms belong to $A\otimes J_A$ and $J_A\otimes A$. We get from Lemma \ref{Lem:R-Ind-der} that the second term is in $A\otimes J_A$, hence $\dgal{A,J_A}\subset A\otimes J_A+J_A\otimes A$. This is also the case for $\dgal{J_A,A}$ by cyclic antisymmetry \eqref{Eq:Dbr-cyc}, hence we get a well-defined map $\dgal{-,-}:A^{\ab}\times A^{\ab}\to A^{\ab}\otimes A^{\ab}$. 
\end{proof}
Note that the double bracket induced from $A$ to $A^{\ab}$ preserves the property of being (weak) Poisson.

\medskip

Let us now recast a well-known result within this formalism. Given $a\in A$, we introduce the notation $a_1=a\otimes 1\otimes 1$, and we write in the same way $a_2,a_3$ (where the subscript refer to the position of $a$ in $A^{\otimes 3}$). 
For an element $r\in A^{\otimes 2}$, let us introduce the following embeddings into $A^{\otimes 3}$:
$$r_{12}=r'\otimes r'' \otimes 1,\quad 
r_{13}=r'\otimes 1 \otimes r'',\quad 
r_{23}=1\otimes r'\otimes r''\,.$$
We also let $r_{ij}=(r^\circ)_{ji}$ when $i>j$. 
We say that $r$ is a \emph{solution to the Classical Yang-Baxter Equation}, or \emph{CYBE} for short, if (we do not require $r=-r^\circ$) 
\begin{equation}\label{Eq:CYBE}
    [r_{12},r_{13}]+[r_{12},r_{23}]+[r_{32},r_{13}]=0\,.
\end{equation}
Here, we use the commutator on $A^{\otimes 3}$.
For the next statement, we consider $B$ to be an  algebra containing a complete set of orthogonal idempotents $(e_i)_{i\in I}$: $e_ie_j=\delta_{ij}e_j$ and $\sum_{i\in I}e_i=1$. We assume that $A$ is freely generated over $B$ by elements $(v_{ij})_{i,j\in I}$ (some $v_{ij}$ could be taken to be zero) satisfying $v_{ij}=e_i v_{ij} e_j$.  
We set $v=\sum_{i,j\in I}v_{ij}$.

\begin{proposition} \label{Pr:Rdbr-CYB}
Assume that $r\in B^{\otimes 2} \subset A^{\otimes 2}$ is a solution to the CYBE. 
Then there is a unique double $(12)$-weak Poisson bracket (associated with the right bimodule structure on $A^{\otimes 2}$) which is $B$-linear and which satisfies 
\begin{equation} \label{Eq:dg-vvr}
    \dgal{v,v}=[r,v_1] -[r^\circ,v_2]\,.
\end{equation}
\end{proposition}
\begin{proof}
Using \eqref{Eq:dg-vvr} and $B$-linearity, we have that 
$\dgal{v_{ij},v_{kl}}=(e_i\otimes e_k)\dgal{v,v}(e_j\otimes e_l)$.  
Thus the operation $\dgal{-,-}$ is defined on all generators and it is only nonzero on the free generators $(v_{ij})$, so it extends uniquely by the Leibniz rules \eqref{Eq:DbrR-Der1}--\eqref{Eq:DbrR-Der2} (with $\alpha=\beta=\id_A$). 
We still have to check the cyclic antisymmetry $\dgal{v,v}=-\dgal{v,v}^\circ$, which easily follows from \eqref{Eq:dg-vvr}. 

To have a  double $(12)$-weak Poisson bracket, it suffices to check the vanishing of the double $(12)$-weak Jacobiator on generators in view of Lemma \ref{Lem:R-wJac}. Since the double bracket is zero when one argument is in $B$, it is in fact sufficient to check 
\begin{equation}
    \dgal{v,v,v}-\tau_{(12)} \dgal{v,v,v}=0\,.
\end{equation}
We first need to compute $\dgal{v,v,v}=(1+\tau_{(123)}+\tau_{(132)})\dgal{v,\dgal{v,v}}_L$. We have that 
\begin{align*}
\dgal{v,\dgal{v,v}}_L=&\dgal{v,r'v-vr'}\otimes r''
=\dgal{v,v}'\otimes [r',\dgal{v,v}'']\otimes r'' \\
=&[r_{23},\dgal{v,v}\otimes 1]
=[r_{23}, [r_{12},v_1]-[r_{21},v_2]]\,.
\end{align*}
This gives for the double $(12)$-weak Jacobiator 
\begin{align*}
    &\dgal{v,v,v}-\tau_{(12)} \dgal{v,v,v} \\
=&(1+\tau_{(123)}+\tau_{(132)})
\Big([r_{23}, [r_{12},v_1]]
-[r_{12}, [r_{13},v_1]] 
+[r_{13}, [r_{12},v_1]]
-[r_{32}, [r_{13},v_1]]  \Big)\,.
\end{align*}
Let us check that all the commutators involving $v_1$ cancel out. 
We note that $[r_{ij},v_1]=0$ whenever both $i,j\neq 1$. 
Since the commutator satisfies Jacobi identity, we can write 
\begin{align*}
\mathtt{S}:=&[r_{23}, [r_{12},v_1]]
-\big([r_{12}, [r_{13},v_1]] 
+[r_{13}, [v_1,r_{12}]]\big)
-[r_{32}, [r_{13},v_1]] \\
=&-\big([r_{12},[v_1,r_{23}]]+[v_1,[r_{23}, r_{12}]]\big)
+[v_1,[r_{12},r_{13}]]
+\big([r_{13}, [v_1,r_{32}]]+[v_1,[r_{32}, r_{13}]\big) \\
=&[v_1,[r_{12},r_{13}]+[r_{12},r_{23}]+[r_{32},r_{13}]]\,.
\end{align*}
Using the CYBE \eqref{Eq:CYBE}, $\mathtt{S}=0$ as desired. 
\end{proof}

\begin{remark}
Proposition \ref{Pr:Rdbr-CYB} formalises a classical result related to $r$-matrices \cite{BBT}; it recently appeared in some particular noncommutative case in the work of Arthamonov, Ovenhouse and Shapiro \cite[\S4.1]{AOS}.
To get examples, we can take $B$ to be the matrix algebra $B=\oplus_{i,j\in I}\kk e_{ij}$, $I=\{1,\ldots,N\}$, with generators subject to $e_{ij}e_{kl}=\delta_{kj}e_{il}$ and where we set $e_i:=e_{ii}$ for the orthogonal idempotents. We can then pick standard solutions of the CYBE and work on $A=\kk[v]\otimes B$, e.g. $r=\sum_{i<j}e_{ij}\otimes e_{ji}+\frac12 \sum_i e_i \otimes e_i$ (which is skewsymmetric after correction by half the Casimir element $\sum_{i,j\in I}e_{ij}\otimes e_{ji}$).   
\end{remark}

\subsection{... from the left bimodule structure} \label{ss:E-Left}

A double bracket $\dgal{-,-}$ on $A$ which is associated with the $(\alpha,\beta)$-twisted  left bimodule structure \eqref{Eq:Mod-Ltwis}  satisfies the Leibniz rules  
\begin{align} \label{Eq:DbrL-Der1}
\dgal{a,bc}=\alpha(b) \dgal{a,c}'\otimes \dgal{a,c}'' 
+ \dgal{a,b}' \beta(c) \otimes \dgal{a,b}'' \,, \\
\label{Eq:DbrL-Der2}
    \dgal{ac,b}= \dgal{c,b}'\otimes \alpha(a)\dgal{c,b}'' 
+ \dgal{a,b}' \otimes \dgal{a,b}'' \beta(c)\,. 
\end{align}
Each Leibniz rule can be recovered from the other one thanks to the cyclic antisymmetry \eqref{Eq:Dbr-cyc}. 
By Lemma \ref{Lem:Ind-Equiv}, the swap automorphism ${}^\circ$ induces an equivalence with a double bracket associated with the $(\alpha,\beta)$-twisted \emph{right} bimodule structure. 
Hence the results covered in \ref{ss:E-Right} in the right case admit an analogous derivation.  

Recall from Lemma \ref{Lem:R-wJac} that for double brackets associated with the right bimodule structure ($\alpha=\beta=\id_A$), the $(12)$-weak double Jacobiator is a derivation in each argument. In particular, the important objects in that case are double $(12)$-weak Poisson brackets. We get the next result from Proposition \ref{Pr:Op-Equ-wk1}.$(i)$.  

\begin{theorem} \label{Thm-Equiv-LR}
A double bracket $\dgal{-,-}$ on $A$ associated with the left bimodule structure is $[(12),(13)]$-weak Poisson if and only if the equivalent double bracket $\dgal{-,-}^\circ$ associated with the right bimodule structure is $(12)$-weak Poisson.
\end{theorem}

\begin{lemma} \label{Lem:L-wJac}
For $\sigma=(12)$ and $\sigma'=(13)$, 
The $[\sigma,\sigma']$-weak double Jacobiator \eqref{Eq:wdJac2} of a double bracket $\dgal{-,-}$ associated with the left bimodule structure is a derivation in each of its entries. 
In that case the third entry acts as the following derivation $A\to A^{\otimes 3}$: for all $a,b,c_1,c_2\in A$,  
$$[\sigma,\sigma']\wk\!\dgal{a,b,c_1c_2}=(c_1 \otimes 1\otimes 1)\,\, [\sigma,\sigma']\wk\!\dgal{a,b,c_2} 
+ [\sigma,\sigma']\wk\!\dgal{a,b,c_1} \,(c_2 \otimes 1\otimes 1).$$  
\end{lemma} 
\begin{proof}
This basically follows from Lemma \ref{Lem:R-wJac} and item $(i)$ in Proposition \ref{Pr:Op-Equ-wk1}. Let us nevertheless make the calculation to see why we need the $[(12),(13)]$-weak double Jacobiator. 
A short computation yields that the double Jacobiator  \eqref{Eq:dJac} satisfies for any $a,b,c_1,c_2\in A$
\begin{align*}
\dgal{a,b,c_1c_2}=&
\,\,(c_1\otimes 1 \otimes 1) \dgal{a,b,c_2}
+ \dgal{a,b,c_1} (c_2\otimes 1 \otimes 1) \\
&+\dgal{a,c_1}' \dgal{b,c_2}'\otimes \dgal{a,c_1}'' \otimes \dgal{b,c_2}'' 
+ \dgal{b,c_1}' \dgal{a,c_2}'\otimes \dgal{a,c_2}'' \otimes \dgal{b,c_1}''\,.
\end{align*}
We can then see that $\dgal{-,-,-}-\tau_{(23)}\dgal{-,-,-}\tau_{(12)}$ is a derivation in its third entry. 
Therefore, we have established the statement for the $[(23),(12)]$-weak double Jacobiator which, due to the cyclic symmetry \eqref{Eq:dJCyc}, is the same as the $[(12),(13)]$-weak double Jacobiator. 
\end{proof}


\begin{lemma} \label{Lem:L-Ind-der}
If $\dgal{-,-}$ is associated with the $\alpha$-twisted left bimodule structure (when $\beta=\alpha$), the double bracket descends to maps  
\begin{equation} \label{Eq:in-brBullet-L}
    {}_\bullet\!\dgal{-,-} : A/[A,A]\times A\to A \otimes A/[A,A]\,, \quad 
    \dgal{-,-}_\bullet :A\times A/[A,A]  \to   A/[A,A]\otimes A\,.
\end{equation}
These operations satisfy  $\alpha$-twisted derivation rules, which are obtained by inducing on $A/[A,A]\otimes A$ (resp. $A\otimes A/[A,A]$) the $\alpha$-twisted right (resp. left) bimodule structure on $A^{\otimes 2}$.  
Furthermore, both maps descend to the same operation  
${}_\bullet\!\dgal{-,-}_\bullet :A/[A,A]\times A/[A,A]\to A/[A,A]\otimes A/[A,A]$.
\end{lemma}

\begin{proposition} \label{Pr:L-Ind}
Let $\dgal{-,-}$ be a double $[(12),(13)]$-weak Poisson bracket associated with the left bimodule structure.  
Then the operation ${}_\bullet\!\dgal{-,-}_\bullet$ from Lemma \ref{Lem:L-Ind-der} uniquely extends  to a Poisson bracket on $\Sym(A/[A,A])$.  
\end{proposition}
\begin{proof}
Using the notations from the proof of Proposition \ref{Pr:R-Ind}, we define the Poisson bracket on $\Sym(A/[A,A])$ as the unique bilinear operation satisfying $\br{a,b}_{\Sym}=\pi_2 \dgal{a,b}$ for $a,b\in A/[A,A]$ (after taking lifts $a,b\in A$). Similarly to that proof, to check Jacobi identity we compute  
\begin{equation} \label{Eq:L-Ind}
\br{a,\br{b,c}_{\Sym}}_{\Sym} + \br{b,\br{c,a}_{\Sym}}_{\Sym}
+ \br{c,\br{a,b}_{\Sym}}_{\Sym} 
=\pi_3(\dgal{a,b,c}-\dgal{b,a,c})\,.
\end{equation}
By cyclic symmetry \eqref{Eq:dJCyc}, this equals 
$\pi_3(\dgal{a,b,c}-\tau_{(132)}\dgal{c,b,a})$. Since $\pi_3\circ \tau_\sigma=\pi_3$ for any $\sigma\in S_3$, we can write \eqref{Eq:L-Ind} as 
$\pi_3(\dgal{a,b,c}-\tau_{(12)}\dgal{c,b,a})
= \pi_3\,([\sigma,\sigma']\wk\!\dgal{a,b,c})$, 
which is zero by assumption. 
\end{proof}

\begin{theorem} \label{Thm-LR}
Let $\dgal{-,-}_1$ and $\dgal{-,-}_2$ be double brackets associated with the left and right bimodule structures, respectively. 
Assume that they are equivalent under the swap automorphism, that is $\dgal{-,-}_1=\dgal{-,-}_2^\circ$. 
If either $\dgal{-,-}_1$ is $[(12),(13)]$-weak Poisson or if $\dgal{-,-}_2$ is $(12)$-weak Poisson, then the two double brackets induce the same Poisson bracket on $\Sym(A/[A,A])$.
\end{theorem}
\begin{proof}
By Proposition \ref{Pr:Op-Equ-wk1}.$(i)$, $\dgal{-,-}_1$ is $[(12),(13)]$-weak Poisson if and only if  $\dgal{-,-}_2$ is $(12)$-weak Poisson. Therefore, by Propositions \ref{Pr:R-Ind} (in the right case) and \ref{Pr:L-Ind} (in the left case), we get that the maps $\br{-,-}_{j,\Sym}$ on $\Sym(A/[A,A])$, which are uniquely defined to satisfy $\br{a,b}_{j,\Sym}:=\pi_2\dgal{a,b}_j$, are Poisson brackets. 
So we are left to prove $\br{a,b}_{1,\Sym}=\br{a,b}_{2,\Sym}$. 
This follows from   
$$\br{a,b}_{1,\Sym} = \pi_2 \dgal{a,b}_1=\pi_2 \dgal{a,b}_2^\circ
=  \pi_2\circ \tau_{(12)} \dgal{a,b}_2=\pi_2\dgal{a,b}_2=\br{a,b}_{2,\Sym} \,,$$
using the equivalence of the two double brackets and the fact that $\pi_2\circ \tau_\sigma=\pi_2$ for any $\sigma\in S_2$. 
\end{proof}

\begin{lemma} \label{Lem:Dbr-Ab-L}
If $\dgal{-,-}$ is associated with the $\alpha$-twisted left bimodule structure, then the double bracket descends to $A^{\ab}$. 
Furthermore, this double bracket induced on $A^{\ab}$ is equivalent under the swap automorphism to the double bracket (associated with the $\alpha$-twisted right bimodule) induced by $\dgal{-,-}^\circ$ through Lemma \ref{Lem:Dbr-Ab-R}.
\end{lemma}
\begin{proof}
This is direct by comparison with Lemma \ref{Lem:Dbr-Ab-R}. 
\end{proof}

In the next result, we use the notations introduced before Proposition \ref{Pr:Rdbr-CYB}. 

\begin{proposition} \label{Pr:Ldbr-CYB}
Assume that $r\in B^{\otimes 2} \subset A^{\otimes 2}$ is a solution to the CYBE.
Then there is a unique double $[(12),(13)]$-weak Poisson bracket (associated with the left bimodule structure on $A^{\otimes 2}$) which is $B$-linear and which satisfies 
$\dgal{v,v}=[r,v_1] -[r^\circ,v_2]$. 
\end{proposition}
\begin{proof}
Setting $\dgal{-,-}_2=\dgal{-,-}^\circ$, $\dgal{v,v}_2$ takes the form \eqref{Eq:dg-vvr} 
(with $-r$ as the solution to the CYBE) so the result follows from  Proposition \ref{Pr:Rdbr-CYB} and Theorem \ref{Thm-Equiv-LR}. 
\end{proof}

\subsection{... related to morphisms} \label{ss:Exmp-Morph}

The (non-twisted) left, right, outer and inner bimodule structures \eqref{Eq:Mod-L}--\eqref{Eq:Mod-in} are well-defined on $A^{\otimes 2}$ for any algebra $A$. In particular, it makes sense to speak about a morphism of double brackets $\phi:(A_1,\dgal{-,-}_1)\to (A_2,\dgal{-,-}_2)$ when both double brackets are associated with the left (resp. right, outer or inner) bimodule structure, even though $A_1$ and $A_2$ could be different. 

\begin{example}
Let $\dgal{-,-}$ be a double bracket on $A$ associated with the left bimodule structure. 
Denote by $\dgal{-,-}^{\ab}$ the double bracket on $A^{\ab}$ induced by Lemma \ref{Lem:Dbr-Ab-L}. 
Then the projection $\pi^{\ab}:A\to A^{\ab}=A/J_A$, where $J_A$ is the (two-sided) ideal generated by $A/[A,A]$, is a morphism of double brackets. Indeed, we have that for any $a,b\in A$, 
$$\dgal{\pi^{\ab}(a),\pi^{\ab}(b)}^{\ab}=(\pi^{\ab}\otimes \pi^{\ab})(\dgal{a,b})\,,$$
by definition of the induced double bracket on $A^{\ab}$. 
\end{example}

Let us now investigate how morphisms interact with the different constructions given in the previous subsections. We only study the cases of double brackets associated with the right or outer bimodules, since they are equivalent to double brackets associated with the left or inner bimodules respectively. 

\begin{lemma}
Let $\phi:(A_1,\dgal{-,-}_1)\to (A_2,\dgal{-,-}_2)$ be a morphism of double brackets associated with the outer bimodule structure. 
\begin{enumerate}
    \item $\br{\phi(a),\phi(b)}_{2,\mult}=\phi(\br{a,b}_{1,\mult})$ for any $a,b\in A$ and $\br{-,-}_{j,\mult}=\mult\circ \dgal{-,-}_j:A_j \times A_j \to A_j$. \newline
Moreover, this identity holds for the linear map $\overline{\phi}:A_1/[A_1,A_1]\to A_2/[A_2,A_2]$ induced by $\phi$ and
the induced antisymmetric maps 
$\br{-,-}_{j,\mult}:A_j/[A_j,A_j] \times A_j/[A_j,A_j] \to A_j/[A_j,A_j]$. 

\item If $\phi$ is a morphism of double Poisson algebras, then $\phi:(A_1,\br{-,-}_{1,\mult})\to (A_2,\br{-,-}_{2,\mult})$ is a morphism of left Loday algebras, while 
$\overline{\phi}:(A_1/[A_1,A_1],\br{-,-}_{1,\mult})\to (A_2/[A_2,A_2],\br{-,-}_{2,\mult})$ is a morphism of Lie algebras. 
\end{enumerate}
\end{lemma}
\begin{proof}
By Proposition \ref{Pr:Out-Lod}, we have in the Poisson case that $(A_j,\br{-,-}_{j,\mult})$ is a left Loday algebra while $(A_j/[A_j,A_j],\br{-,-}_{j,\mult})$ is a Lie algebra, so item \emph{2.} is a consequence of item \emph{1}. 

For the first part of item \emph{1}, we simply need to remark that 
$$\br{\phi(a),\phi(b)}_{2,\mult}
=\dgal{\phi(a),\phi(b)}_{2}'\dgal{\phi(a),\phi(b)}_{2}''
=\mult \circ \phi^{\otimes 2} \dgal{a,b}_1
=\phi(\br{a,b}_{1,\mult})\,.$$
We get the statement for $\overline{\phi}$ because $\phi([A_1,A_1])\subset [A_2,A_2]$ and the induced operation on $A_j/[A_j,A_j]$ is obtained by projecting $\br{-,-}_{j,\mult}$ from $A_j$ to $A_j/[A_j,A_j]$ with $j=1,2$, see Lemma \ref{Lem:Out-Ind-2}. 
\end{proof}

\begin{lemma}
Let $\phi:(A_1,\dgal{-,-}_1)\to (A_2,\dgal{-,-}_2)$ be a morphism of double brackets associated with the right bimodule structure. 
\begin{enumerate}
    \item Define the maps ${}_\bullet\!\dgal{-,-}_j$ and $\dgal{-,-}_{j,\bullet}$ from $\dgal{-,-}_j$ as in \eqref{Eq:in-brBullet}, and consider the linear map $\overline{\phi}:A_1/[A_1,A_1]\to A_2/[A_2,A_2]$ induced by $\phi$.  
Then ${}_\bullet\!\dgsmall{\overline{\phi}(\bar{a}),\phi(b)}_2
    =(\overline{\phi}\otimes \phi)({}_\bullet\!\dgal{\bar{a},b}_1)$ and 
$\dgsmall{\phi(b),\overline{\phi}(\bar{a})}_{2,\bullet}
    =(\phi \otimes \overline{\phi})(\dgal{b,\bar{a}}_{1,\bullet})$ for any $\bar{a}\in A/[A,A]$, $b\in A$. 
    
\item Define the maps ${}_\bullet\!\dgal{-,-}_{j,\bullet}$  as in Lemma \ref{Lem:R-Ind-der}. 
Then ${}_\bullet\!\dgsmall{\overline{\phi}(\bar{a}),\overline{\phi}(\bar{b})}_{2,\bullet}
    =\overline{\phi}^{\otimes 2}\,({}_\bullet\!\dgal{\bar{a},\bar{b}}_{1,\bullet})$
for any $\bar{a},\bar{b}\in A/[A,A]$.

\item If $\phi$ is a morphism of double $(12)$-weak Poisson algebras, and we let 
$\br{-,-}_{j,\Sym}$ be the unique Poisson bracket extending ${}_\bullet\!\dgal{-,-}_{j,\bullet}$ on $\Sym(A_j/[A_j,A_j])$, 
then $\overline{\phi}$ uniquely extends to a morphism of Poisson algebras 
$\phi_{\Sym}:\Sym(A_1/[A_1,A_1])\to \Sym(A_2/[A_2,A_2])$. 
\end{enumerate}
\end{lemma}
\begin{proof}
For item \emph{1.}, we have by definition that 
\begin{align*}
{}_\bullet\!\dgsmall{\overline{\phi}(\bar{a}),\phi(b)}_2    
=&\big(\dgal{\phi(a),\phi(b)}_2'+ [A_2,A_2]\big) \otimes \dgal{\phi(a),\phi(b)}_2'' \\
=&\big(\phi(\dgal{a,b}_1')+ [A_2,A_2] \big) \otimes \phi(\dgal{a,b}_1'') 
=(\overline{\phi}\otimes \phi)({}_\bullet\!\dgal{\bar{a},b}_1)\,.
\end{align*}
The second identity is proved similarly. Item \emph{2.} is also obtained in that way. 
Finally, introduce the map $\phi_{\Sym}:\Sym(A_1/[A_1,A_1])\to \Sym(A_2/[A_2,A_2])$ through  $\phi_{\Sym}(a_1 \bullet \ldots \bullet a_n)=\overline{\phi}(a_1)\bullet \ldots \bullet \overline{\phi}(a_n)$ with the notations of the proof of Proposition \ref{Pr:R-Ind}. 
Then, item \emph{3.} is a consequence of the uniqueness of $\br{-,-}_{j,\Sym}$ by extension through Leibniz rules of its definition on elements of $A_j/[A_j,A_j]$.
\end{proof}


\section{Operations induced on representation spaces}  \label{S:Rep}

In this section, we fix a positive integer $n\in \N^\times$. 
We denote by $\Mat_n(B)$ the ring of $n\times n$ matrices with coefficients in a $\kk$-algebra $B$.

Given a $\kk$-algebra $A$, we introduce the $n$-th representation space $\Rep(A,n)$ as the set which parametrises representations $\rho$ of $A$ on $\kk^n$, i.e. for any $a\in A$ we have  $\rho(a)\in \Mat_{n}(\kk)$. 
We also let $\Xtt(a)$, for $a\in A$, denote the matrix-valued function on $\Rep(A,n)$ defined through $\Xtt(a)(\rho)=\rho(a)$. 
For $1\leq i,j\leq n$, we simply denote by $a_{ij}:=\Xtt(a)_{ij}$ the function returning the $(i,j)$ entry of $\rho(a)$ at $\rho\in \Rep(A,n)$. 
In that way, we get a ring of functions on $\Rep(A,n)$ generated by symbols $a_{ij}$ with $a\in A$ and $1\leq i,j\leq n$ subject to the following conditions of (matrix) addition and multiplication: 
\begin{equation*}
    \lambda a_{ij}+\mu b_{ij}=(\lambda a+\mu b)_{ij},\quad 
1_{ij}=\delta_{ij}, \quad (ab)_{ij}=\sum_{1\leq k \leq n}a_{ik}b_{kj}\,,
\end{equation*}
where $a,b\in A$, $\lambda,\mu \in \kk$, $1\leq i,j\leq n$.
We denote this commutative ring $\kk[\Rep(A,n)]$ and, to ease notations, we also write $\Rep(A,n)$ for the corresponding affine scheme. 
It is well-known that $\Rep(A,n)$ does not depend on the presentation of $A$ because it represents the functor which sends a finitely generated $\kk$-algebra $B$ 
to the set of ($\kk$-algebra) homomorphisms $\Hom(A, \Mat_n(B))$.  
Let us also note that any homomorphism $\phi:A_1\to A_2$ induces a homomorphism $\phi_n:\kk[\Rep(A_1,n)]\to \kk[\Rep(A_2,n)]$ obtained by setting $\phi_n(a_{ij}):=\phi(a)_{ij}$ for any $a\in A_1$. In matrix form, we set  $\phi_n(\Xtt(a)):=\Xtt(\phi(a))$.  
Finally, we introduce some notation: given $d\in A^{\otimes 2}$, $e\in A^{\otimes 3}$ 
and indices $i,j,k,l,u,v\in \{1,\ldots, n\}$, we set 
$$d_{ij,kl}:= d'_{ij}\,d''_{kl}\,, \qquad
e_{ij,kl,uv}:=e'_{ij} \, e''_{kl} \, e'''_{uv}\,,$$
which are elements of $\kk[\Rep(A,n)]$.

\begin{remark}
The different constructions given below can be adapted to double brackets that are $B$-linear, see \ref{ss:base}. In that situation, one has to consider the space $\Rep_B(A,n)$ of representations \emph{relative to $B$}. 
For the important case $B=\oplus_{i \in I}\kk e_i$, the explicit construction can be found in \cite[\S7.1]{VdB1}. 
\end{remark}

\subsection{Case of the outer bimodule structure} 

We consider double brackets on $A$ associated with the $\alpha$-twisted outer bimodule structure on $A^{\otimes 2}$, see \ref{ss:E-Out} with $\beta=\alpha$. Recall that for any $n\geq 1$,  $\alpha_n$ denotes the automorphism of $\kk[\Rep(A,n)]$ induced by $\alpha$. 
The following result is a  straightforward generalisation of Van den Bergh's work \cite[\S7.5]{VdB1}. 

\begin{proposition} \label{Pr:Rep-twOut}
If $\dgal{-,-}$ is a double bracket associated with the $\alpha$-twisted outer bimodule structure, 
there is a unique antisymmetric bilinear operation $\br{-,-}$ on $\kk[\Rep(A,n)]$  
which is a $\alpha_n$-twisted biderivation, i.e. for any $f,g,h\in \kk[\Rep(A,n)]$
$$\br{f,gh}=\alpha_n(g) \br{f,h}+ \br{f,g} \alpha_n(h), \quad 
\br{fh,g}=\alpha_n(f) \br{h,g}+ \br{f,g} \alpha_n(h)\,,$$
and which satisfies for any $a,b\in A$, 
\begin{equation} \label{Eq:Rep-twOut}
    \br{a_{ij},b_{kl}}=\dgal{a,b}_{kj,il}=\dgal{a,b}'_{kj}\dgal{a,b}''_{il}\,.
\end{equation}
\end{proposition}
\begin{proof}
Uniqueness is clear, if the bilinear operation is well-defined. To check the latter, note that since $\dgal{-,-}$ is associated with the $\alpha$-twisted outer bimodule we have from \eqref{Eq:Rep-twOut}   
\begin{align*}
    \br{a_{ij},(bc)_{kl}}=&(\alpha(b)\dgal{a,c}'\otimes \dgal{a,c}''+\dgal{a,b}'\otimes \dgal{a,b}''\alpha(c))_{kj,il} \\
=&\sum_{1\leq u \leq n} (\,\alpha(b)_{ku} \dgal{a,c}_{uj,il}+ \dgal{a,b}_{kj,iu} \alpha(c)_{ul}\,) \\
=&\sum_{1\leq u \leq n} (\,\alpha_n(b_{ku}) \br{a_{ij},c_{ul}} + \br{a_{ij},b_{ku}} \alpha_n(c_{ul})\,) \,.
\end{align*}
This is precisely $\br{a_{ij},\sum_u b_{ku} c_{ul}}$ thanks to the $\alpha_n$-twisted biderivation rules. 
The antisymmetry of $\br{-,-}$ follows from the cyclic antisymmetry of $\dgal{-,-}$ as noted in \cite[Prop. 7.5.1]{VdB1}. 
\end{proof}

\begin{remark} \label{Rem:RepMat-Out}
When $\alpha=\id_A$, we can encode the operation $\br{-,-}$ satisfying \eqref{Eq:Rep-twOut} in terms of an operation on matrices $\Xtt(a)=(a_{ij})$, $a\in A$, defined through 
$$\dgal{\Xtt(a),\Xtt(b)}_{(n)}
:=\sum_{1\leq i,j,k,l\leq n} 
\br{a_{ij},b_{kl}} \, E_{kj}\otimes E_{il}\,,$$
for any $a,b \in A$. This gives precisely the operation 
\eqref{Eq:Br-VdB} from the introduction, which satisfies the antisymmetry and Leibniz rules \eqref{Eq:Intro-VdB}.
\end{remark}

In contrast to the case of Proposition \ref{Pr:Rep-twOut}, an arbitrary double bracket may \emph{not} yield a well-defined map on representation spaces. It can thus be thought of as a purely noncommutative operation, as the next example shows. 

\begin{example} 
Let $A=\kk\langle x,y\rangle$, and consider the one dimensional representation  $\kk[Rep(A,1)]\simeq \kk[\hat{x},\hat{y}]$ with $\hat{x}=x_{11}$, $\hat{y}=y_{11}$. We take $\alpha = \id_A$, $\beta: x \mapsto y, y \mapsto x$, and we also denote by $\alpha,\beta$ the induced maps on $\kk[\hat{x},\hat{y}]$. 
Consider the unique double bracket associated with the $(\alpha,\beta)$-twisted outer bimodule such that $\dgal{x,x}=0=\dgal{y,y}$ and $\dgal{x,y}\neq 0$.
Assume that there is a well-defined operation $\br{-,-}$ such that \eqref{Eq:Rep-twOut} holds. 
As $\hat{x}\hat{y}$ corresponds to $(xy)_{11}$, we calculate from the Leibniz rules of $\dgal{-,-}$ that  
\begin{align*}
    \{\hat{x}\hat{y},\hat{y}\} &= \alpha(\hat{x}) \{\hat{y},\hat{y}\} + \{\hat{x},\hat{y}\}\beta(\hat{y})=\hat{x}\{\hat{x},\hat{y}\} \,, \\
\{\hat{x}\hat{y},\hat{y}\} &=\{\hat{y}\hat{x},\hat{y}\} = \alpha(\hat{y}) \{\hat{x},\hat{y}\} + \{\hat{y},\hat{y}\}\beta(\hat{x})=\hat{y}\{\hat{x},\hat{y}\} \,,
\end{align*}
from which $(\hat{x}-\hat{y})\{\hat{x},\hat{y}\}=0$.
If $\dgal{x,y}_{11,11}\neq 0$ on $\Rep(A,1)$  (e.g. $\dgal{x,y}=1\otimes 1$), we have under the above identification that 
$\{\hat{x},\hat{y}\}=\dgal{x,y}_{11,11}$ is nonzero. So our observations yield that $\hat{x}-\hat{y}=0$. This is not a relation in $\kk[\hat{x},\hat{y}]$, so $\br{-,-}$  does not exist. 
\end{example}

Recall from Example \ref{Ex:ctr-Equiv} that a double bracket $\dgal{-,-}$ associated with the $\alpha$-twisted outer bimodule is equivalent to the double bracket $(\alpha^{-1})^{\otimes 2}\circ\dgal{-,-}$ associated with the (non-twisted) outer bimodule. On representation spaces, this means that the operation $\br{-,-}$ from Proposition \ref{Pr:Rep-twOut} can be equivalently understood from $\alpha^{-1}_n\circ\br{-,-}$, which is an antisymmetric biderivation. Since interesting examples of double Poisson brackets only appear for  $\alpha=\id_A$, we restrict to that case where we recall the following result.

\begin{theorem}[\cite{VdB1}] \label{Thm:Rep-Out-Lie} 
If $\dgal{-,-}$ is a double Poisson bracket associated with the outer bimodule structure, 
then the antisymmetric biderivation $\br{-,-}$ on $\kk[\Rep(A,n)]$ given in Proposition \ref{Pr:Rep-twOut} is a Poisson bracket. 
Furthermore, the Poisson bracket descends to the coordinate ring of trace functions 
$\kk[\Rep(A,n)]^{\tr}$ generated by $\{\tr\Xtt(a) \mid a \in A\}$. 
\end{theorem}
\begin{proof}
The first part is Proposition 7.5.2 in \cite{VdB1}. For comparison with other cases, we note that this result follows from the identity  
\begin{equation}
    \br{a_{ij} , \br{ b_{kl} , c_{uv} }}
+ \br{b_{kl} , \br{ c_{uv} , a_{ij} }}
+\br{c_{uv} , \br{ a_{ij}, b_{kl} }} 
=
\dgal{a,b,c}_{uj,il,kv} - \dgal{a,c,b}_{kj,iv,ul}\,,
\end{equation}
for any $a,b,c\in A$ and $i,j,k,l,u,v\in \{1,\ldots,n\}$. 
The second part is Proposition 7.7.2 in \cite{VdB1} (see more generally \cite{CB11}), which is obtained from the first part and the fact that \eqref{Eq:Rep-twOut} yields $\br{\tr\Xtt(a),\tr\Xtt(b)}=\tr\Xtt(\br{a,b}_{\mult})$, where $\br{-,-}_{\mult}$ is given by \eqref{Eq:Out-brm}. 
\end{proof}

\begin{remark}
When $\kk$ is algebraically closed in addition to $\operatorname{char}(\kk)=0$, 
$\kk[\Rep(A,n)]^{\tr}$ is the coordinate ring of the GIT quotient $\Rep(A,n)/\!/\Gl_n(\kk)$. 
Here, $\Gl_n(\kk)$ acts by conjugation on each $\Xtt(a)$, $a\in A$. 
(This follows from the celebrated Le Bruyn-Procesi theorem \cite{LP90} and it also holds in the relative semisimple case, see \cite{CB11}.)
In that case, Theorem \ref{Thm:Rep-Out-Lie} is therefore turning $\Rep(A,n)/\!/\Gl_n(\kk)$ into a Poisson variety. This comment also applies for the analogous theorems given below. 
\end{remark}

\subsection{Case of the inner bimodule structure}

We now work with double brackets on $A$ associated with the $\alpha$-twisted inner bimodule structure on $A^{\otimes 2}$, see \ref{ss:E-Int} with $\beta=\alpha$. 
\begin{proposition} \label{Pr:Rep-twIn}
If $\dgal{-,-}$ is a double bracket associated with the $\alpha$-twisted inner bimodule structure,
there is a unique antisymmetric bilinear operation $\br{-,-}$ on $\kk[\Rep(A,n)]$ 
which is a $\alpha_n$-twisted biderivation 
satisfying for any $a,b\in A$, 
\begin{equation} \label{Eq:Rep-twIn}
    \br{a_{ij},b_{kl}}=\dgal{a,b}_{il,kj}=\dgal{a,b}'_{il}\dgal{a,b}''_{kj}\,.
\end{equation}
\end{proposition}
Note that this result is a reformulation of Proposition \ref{Pr:Rep-twOut} using the equivalence induced by the swap automorphism $^\circ$ between double brackets associated with the $\alpha$-twisted inner/outer bimodule. 
Since this equivalence extends to the Poisson case (see Theorem \ref{Thm:Poi-OutIn}), we can expect the following analogue of Theorem \ref{Thm:Rep-Out-Lie}.

\begin{theorem} \label{Thm:Rep-In-Lie} 
If $\dgal{-,-}$ is a double Poisson bracket associated with the inner bimodule structure, 
then the antisymmetric biderivation $\br{-,-}$ on $\kk[\Rep(A,n)]$ given in Proposition \ref{Pr:Rep-twIn} is a Poisson bracket. 
Furthermore, the Poisson bracket descends to the ring of trace functions 
$\kk[\Rep(A,n)]^{\tr}$. 
\end{theorem}
\begin{proof}
This basically follows from Theorem \ref{Thm:Rep-Out-Lie}. Without relying on that result, we compute 
\begin{align*}
\br{a_{ij},\br{b_{kl},c_{uv}}}=& \br{a_{ij}, \dgal{b,c}_{kv,ul}} 
=\dgal{a,\dgal{b,c}'}_{iv,kj} \dgal{b,c}''_{ul} 
+ \dgal{a,\dgal{b,c}''}_{il,uj} \dgal{b,c}'_{kv} \\
=&\left(\dgal{a,\dgal{b,c}}_L \right)_{iv,kj,ul}
- \left( \dgal{a,\dgal{c,b}}_L \right)_{il,uj,kv}\,,
\end{align*}
thanks to \eqref{Eq:Rep-twIn} and the cyclic antisymmetry \eqref{Eq:Dbr-cyc}. 
This can be used directly to write  
\begin{align*}
\br{b_{kl},\br{c_{uv},a_{ij}}}
=&\left(\tau_{(123)}\dgal{b,\dgal{c,a}}_L \right)_{iv,kj,ul}
- \left(\tau_{(132)} \dgal{b,\dgal{a,c}}_L \right)_{il,uj,kv}\,, \\
\br{c_{uv},\br{a_{ij},b_{kl}}}
=&\left(\tau_{(132)}\dgal{c,\dgal{a,b}}_L \right)_{iv,kj,ul}
- \left(\tau_{(123)} \dgal{c,\dgal{b,a}}_L \right)_{il,uj,kv}\,.
\end{align*}
Hence we obtain for any $a,b,c\in A$ and indices in $\{1,\ldots,n\}$, 
\begin{equation}
    \br{a_{ij} , \br{ b_{kl} , c_{uv} }}
+ \br{b_{kl} , \br{ c_{uv} , a_{ij} }}
+\br{c_{uv} , \br{ a_{ij}, b_{kl} }} 
=\dgal{a,b,c}_{iv,kj,ul} - \dgal{a,c,b}_{il,uj,kv}\,,
\end{equation}
which concludes the first part. 
The second part is then obtained by taking the trace in \eqref{Eq:Rep-twIn}. 
\end{proof}

\subsection{Case of the right bimodule structure}

In this subsection, we consider double brackets on $A$ associated with the $\alpha$-twisted right bimodule structure on $A^{\otimes 2}$, see \ref{ss:E-Right} with $\beta=\alpha$.  
\begin{proposition} \label{Pr:Rep-twR}
If $\dgal{-,-}$ is a double bracket associated with the $\alpha$-twisted right bimodule structure, 
there is a unique antisymmetric bilinear operation $\br{-,-}$ on $\kk[\Rep(A,n)]$  
which is a $\alpha_n$-twisted biderivation   satisfying for any $a,b\in A$, 
\begin{equation} \label{Eq:Rep-twR}
    \br{a_{ij},b_{kl}}=\dgal{a,b}_{ij,kl}=\dgal{a,b}'_{ij}\dgal{a,b}''_{kl}\,.
\end{equation}
\end{proposition}
\begin{proof}
Similar to the proof of Proposition \ref{Pr:Rep-twOut}.
\end{proof}

\begin{remark} \label{Rem:RepMat-Right}
When $\alpha=\id_A$, we can encode the operation $\br{-,-}$ satisfying \eqref{Eq:Rep-twR} in terms of an operation on matrices $\Xtt(a)=(a_{ij})$, $a\in A$, defined through 
$$
\br{\Xtt(a)\stackrel{\otimes}{,}\Xtt(b)}
:=\sum_{1\leq i,j,k,l \leq n} 
\br{a_{ij},b_{kl}} \, E_{ij}\otimes E_{kl}
=\Xtt(\dgal{a,b}')\otimes \Xtt(\dgal{a,b}'')\,,$$
for any $a,b \in A$. This gives precisely the operation 
\eqref{Eq:Br-tens} from the introduction, which satisfies the antisymmetry 
$\br{\Xtt(a)\stackrel{\otimes}{,}\Xtt(b)}
=-\tau_{(12)}\br{\Xtt(b)\stackrel{\otimes}{,}\Xtt(a)}$ and Leibniz rules 
\eqref{Eq:Intro-tens1}--\eqref{Eq:Intro-tens2}.
\end{remark}

As in the case of the outer bimodule structure, it is possible to induce a Poisson bracket on representation spaces when $\alpha=\id_A$. In this situation, we have to consider  \emph{double $(12)$-weak Poisson brackets} instead of double Poisson brackets. The motivation stems from the fact that for double brackets associated with the right bimodule (cf. \ref{ss:E-Right}), the double Jacobiator does not behave as a derivation in each argument but the $(12)$-weak double Jacobiator does, see Lemma \ref{Lem:R-wJac}. 

\begin{theorem} \label{Thm:Rep-Right-Lie} 
If $\dgal{-,-}$ is a double $(12)$-weak Poisson bracket associated with the right bimodule structure, 
then the antisymmetric biderivation $\br{-,-}$ on $\kk[\Rep(A,n)]$ given in Proposition \ref{Pr:Rep-twR} is a Poisson bracket. 
Furthermore, the Poisson bracket descends to the ring of trace functions 
$\kk[\Rep(A,n)]^{\tr}$. 
\end{theorem}
\begin{proof}
We can directly calculate 
\begin{align*}
\br{a_{ij},\br{b_{kl},c_{uv}}}=& \br{a_{ij}, \dgal{b,c}_{kl,uv}} 
=\dgal{a,\dgal{b,c}'}_{ij,kl} \dgal{b,c}''_{uv} 
+ \dgal{a,\dgal{b,c}''}_{ij,uv} \dgal{b,c}'_{kl} \\
=&\left(\dgal{a,\dgal{b,c}}_L \right)_{ij,kl,uv}
+ \left( \dgal{a,\dgal{b,c}}_R \right)_{kl,ij,uv}\,,
\end{align*}
thanks to \eqref{Eq:Rep-twR} and the map $\dgal{-,-}_R$ \eqref{Eq:RightDmap}. 
This can be used easily to write  
\begin{align*}
\br{b_{kl},\br{c_{uv},a_{ij}}}
=&\left(\tau_{(123)}\dgal{b,\dgal{c,a}}_L \right)_{ij,kl,uv}
+ \left(\tau_{(132)} \dgal{b,\dgal{c,a}}_R \right)_{kl,ij,uv}\,, \\
\br{c_{uv},\br{a_{ij},b_{kl}}}
=&\left(\tau_{(132)}\dgal{c,\dgal{a,b}}_L \right)_{ij,kl,uv}
+ \left(\tau_{(123)} \dgal{c,\dgal{a,b}}_R \right)_{kl,ij,uv}\,.
\end{align*}
Making use of the definition of the double Jacobiator \eqref{Eq:dJac} as well as  \eqref{Eq:dJac-R}, 
we obtain for any $a,b,c\in A$ and indices in $\{1,\ldots,n\}$, 
\begin{equation}
\begin{aligned}
      &  \br{a_{ij} , \br{ b_{kl} , c_{uv} }}
+ \br{b_{kl} , \br{ c_{uv} , a_{ij} }}
+\br{c_{uv} , \br{ a_{ij}, b_{kl} }} \\
=&\dgal{a,b,c}_{ij,kl,uv} - \dgal{b,a,c}_{kl,ij,uv}
=\big(\,(12)\wk\!\dgal{a,b,c}\,\big)_{ij,kl,uv}\,,
\end{aligned}
\end{equation}
which concludes the first part. 
The second part is a consequence of $\br{\tr\Xtt(a),\tr\Xtt(b)}\in \kk[\Rep(A,n)]^{\tr}$, which follows by taking the trace in \eqref{Eq:Rep-twR}. 
\end{proof}

The operation induced on $\kk[\Rep(A,n)]^{\tr}$ by Theorem \ref{Thm:Rep-Right-Lie} is essentially the geometric counterpart of Proposition \ref{Pr:R-Ind}. 
Indeed, recall that the latter defines a Poisson bracket on $\Sym(A/[A,A])$ by extending the operation ${}_\bullet\!\dgal{-,-}_\bullet$ from Lemma \ref{Lem:R-Ind-der}. Meanwhile, we can compute that  
\begin{equation} \label{Eq:PB-R-tr}
    \br{\tr\Xtt(a),\tr\Xtt(b)}=\tr\Xtt(\dgal{a,b}')\,\tr\Xtt(\dgal{a,b}'')\,, \qquad 
    a,b\in A\,.
\end{equation} 
In that expression, the right-hand side can be equivalently obtained by taking trace of the matrices $\Xtt({}_\bullet\!\dgal{a,b}_\bullet')$ and $\Xtt({}_\bullet\!\dgal{a,b}_\bullet'')$ which represent lifts of the two components of ${}_\bullet\!\dgal{a,b}_\bullet$.  Since we take the trace of such matrices, there is no dependence on the choice of lift from $A/[A,A]$ to $A$. 
\begin{remark}
Let us note that Theorem \ref{Thm:Rep-Right-Lie} defines a quadratic Poisson bracket on $\kk[\Rep(A,n)]^{\tr}$ in view of \eqref{Eq:PB-R-tr}.
This contrasts with the linear Poisson bracket defined through  Theorem \ref{Thm:Rep-Out-Lie} in the presence of a double bracket associated with the outer bimodule structure. 
\end{remark}
 
\subsection{Case of the left bimodule structure}

We now work with double brackets on $A$ associated with the $\alpha$-twisted left bimodule structure on $A^{\otimes 2}$, see \ref{ss:E-Left} with $\beta=\alpha$. 
\begin{proposition} \label{Pr:Rep-twL}
If $\dgal{-,-}$ is a double bracket associated with the $\alpha$-twisted left bimodule structure, 
there is a unique antisymmetric bilinear operation $\br{-,-}$ on $\kk[\Rep(A,n)]$  
which is a $\alpha_n$-twisted biderivation 
satisfying for any $a,b\in A$, 
\begin{equation} \label{Eq:Rep-twL}
    \br{a_{ij},b_{kl}}=\dgal{a,b}_{kl,ij}=\dgal{a,b}'_{kl}\dgal{a,b}''_{ij}\,.
\end{equation}
\end{proposition}
This result is equivalent to Proposition \ref{Pr:Rep-twOut} using that the swap automorphism $^\circ$ induces an equivalence between double brackets associated with the $\alpha$-twisted left/right bimodule. 
As the equivalence extends to the weak Poisson case (see Theorem \ref{Thm-Equiv-LR}), we have the following analogue of Theorem \ref{Thm:Rep-Right-Lie}.

\begin{theorem} \label{Thm:Rep-Left-Lie} 
If $\dgal{-,-}$ is a double $[(12),(13)]$-weak Poisson bracket associated with the left bimodule structure, 
then the antisymmetric biderivation $\br{-,-}$ on $\kk[\Rep(A,n)]$ given in Proposition \ref{Pr:Rep-twL} is a Poisson bracket. 
Furthermore, the Poisson bracket descends to the ring of trace functions 
$\kk[\Rep(A,n)]^{\tr}$. 
\end{theorem}
\begin{proof}
In a way similar to the proof of Theorem \ref{Thm:Rep-Right-Lie}, we compute 
\begin{align*}
\br{a_{ij},\br{b_{kl},c_{uv}}}
=&\left(\dgal{a,\dgal{b,c}}_L \right)_{uv,ij,kl}
+ \left( \dgal{a,\dgal{b,c}}_R \right)_{uv,kl,ij}\,.
\end{align*}
thanks to \eqref{Eq:Rep-twL} and the map $\dgal{-,-}_R$ \eqref{Eq:RightDmap}. 
This yields for any $a,b,c\in A$ and indices in $\{1,\ldots,n\}$, 
\begin{equation}
\begin{aligned}
      &  \br{a_{ij} , \br{ b_{kl} , c_{uv} }}
+ \br{b_{kl} , \br{ c_{uv} , a_{ij} }}
+\br{c_{uv} , \br{ a_{ij}, b_{kl} }} \\
=&\dgal{a,b,c}_{uv,ij,kl} - \dgal{b,a,c}_{uv,kl,ij}
=\big(\,[(23),(12)]\wk\!\dgal{a,b,c}\,\big)_{uv,ij,kl}\,,
\end{aligned}
\end{equation}
which vanishes if the $[(23),(12)]$-weak double Jacobiator is zero. As remarked at the end of \ref{ss:wDBR}, this is equivalent to the vanishing of the $[(12),(13)]$-weak double Jacobiator, which is our assumption. 
The second part easily follows. 
\end{proof}

\subsection{Morphisms on representation spaces}

In the situations presented in the previous subsections, 
a morphism of double brackets  can be extended to representation spaces.
As in \ref{ss:Exmp-Morph}, we only consider the cases of  (non-twisted) outer and right bimodule structures; the discussion for the inner and left bimodule structures is a straightforward adaptation of these cases. 

\begin{proposition} \label{Pr:Rep-Out-Mor}
Let $\phi:(A_1,\dgal{-,-}_1)\to (A_2,\dgal{-,-}_2)$ be a morphism of double brackets associated with the outer bimodule structure. 
Denote by $\br{-,-}_j$ the bilinear antisymmetric biderivation induced on $\kk[\Rep(A_j,n)]$ by Proposition \ref{Pr:Rep-twOut}, for  $j=1,2$. 
Then the induced algebra homomorphism $\phi_n:\kk[\Rep(A_1,n)]\to\kk[\Rep(A_2,n)]$ intertwines both operations, i.e. 
$$\br{\phi_n(f),\phi_n(g)}_2=\phi_n(\br{f,g}_1)\,, \quad f,g\in \kk[\Rep(A_1,n)]\,.$$
Furthermore, if both double brackets are Poisson, then $\phi_n$ is a morphism of Poisson algebras.
\end{proposition}
\begin{proof}
Since the operations $\br{-,-}_j$ are antisymmetric biderivations and $\phi_n$ is a morphism, it suffices to show the claim on generators. For any $a,b\in A_1$ and $1\leq i,j,k,l \leq n$, we have by \eqref{Eq:Rep-twOut} and the definition of morphism of double brackets that 
\begin{equation*}
    \br{\phi_n(a_{ij}),\phi_n(b_{kl})}_2
=(\dgal{\phi(a),\phi(b)}_2)_{kj,il}
=\big(\phi^{\otimes 2}(\dgal{a,b}_1)\big)_{kj,il}
=\phi_n \left( \br{ a_{ij} , b_{kl}}_1\right)\,.
\end{equation*}
The second part follows because the operations $\br{-,-}_j$ are Poisson brackets by Theorem \ref{Thm:Rep-Out-Lie}. 
\end{proof}
In the Poisson case, we can go a step further and use Theorem \ref{Thm:Rep-Out-Lie} to obtain a morphism of Poisson algebras $\kk[\Rep(A_1,n)]^{\tr}
\to \kk[\Rep(A_2,n)]^{\tr}$ by restricting $\phi_n$ to trace functions. An alternative proof of this result, based on $H_0$-Poisson structures, can be found in \cite[\S5.1]{F22}.

\begin{proposition} \label{Pr:Rep-Right-Mor}
Let $\phi:(A_1,\dgal{-,-}_1)\to (A_2,\dgal{-,-}_2)$ be a morphism of double brackets associated with the right bimodule structure. Denote by $\br{-,-}_j$ the bilinear antisymmetric biderivation induced on $\kk[\Rep(A_j,n)]$ by Proposition \ref{Pr:Rep-twR}, for $j=1,2$. 
Then the induced algebra homomorphism $\phi_n:\kk[\Rep(A_1,n)]\to\kk[\Rep(A_2,n)]$ intertwines both operations, i.e. 
$$\br{\phi_n(f),\phi_n(g)}_2=\phi_n(\br{f,g}_1)\,, \quad f,g\in \kk[\Rep(A_1,n)]\,.$$
Furthermore, if both double brackets are $(12)$-weak Poisson, then $\phi_n$ is a morphism of Poisson algebras.
\end{proposition}
\begin{proof}
The proof is similar to Proposition \ref{Pr:Rep-Out-Mor}; 
we simply need to use Proposition \ref{Pr:Rep-twR} and Theorem \ref{Thm:Rep-Right-Lie} in the right bimodule case. 
\end{proof}


\section{Comparison between the main double brackets} \label{S:Comp}

It seems important to outline the differences and similarities between the four main sources of examples of double brackets which are associated with the outer, inner, right and left bimodule structures. 
In fact, we have already seen in Section \ref{S:Exmp} that the case of the inner bimodule (resp. the left bimodule) is equivalent to the case of the outer bimodule (resp. the right bimodule). 
Therefore, it only seems necessary to compare the double brackets associated with the outer bimodule structure, and those associated with the right bimodule structure.
We gather their main differences in the following list, where we omit the obvious difference between the Leibniz rules \eqref{Eq:DbrOut-Der1}--\eqref{Eq:DbrOut-Der2} (for the outer case where $\alpha=\beta=\id_A$) and 
\eqref{Eq:DbrR-Der1}--\eqref{Eq:DbrR-Der2} (for the right case where $\alpha=\beta=\id_A$). 

\medskip

\begin{enumerate}
    \item (\emph{Double Jacobi identity}.) 
With the outer bimodule, the double Jacobiator \eqref{Eq:dJac} is a derivation in each argument by Lemma \eqref{Lem:OutJac}.  
With the right bimodule, the $\sigma$-weak double Jacobiator \eqref{Eq:wdJac} with $\sigma=(12)$ is a derivation in each argument by Lemma \ref{Lem:R-wJac}. 
Such derivation rules are \emph{not} satisfied in general by the double Jacobiator in the right bimodule case, or by the $(12)$-weak double Jacobiator in the outer bimodule case. 

\item (\emph{Induced maps}.) 
With the outer bimodule, we get a well-defined map $\br{-,-}_{\mult}:A \times A \to A$ by multiplication of the two factors. This map descends to an antisymmetric operation on $A/[A,A]$, where it becomes a Lie bracket if the double bracket is Poisson, see Lemma \ref{Lem:Out-Ind-2} and Proposition \ref{Pr:Out-Lod}. 
With the right bimodule, we get a well-defined map 
${}_\bullet\!\dgal{-,-}_\bullet:(A/[A,A])^{\times 2} \to (A/[A,A])^{\otimes 2}$ which can be extended to a Poisson bracket on $\Sym(A/[A,A])$ if the double bracket is $(12)$-weak Poisson, see Lemma \ref{Lem:R-Ind-der} and Proposition  \ref{Pr:R-Ind}. 
For nontrivial double brackets, we can not define the operation  $\br{-,-}_{\mult}$ in the right bimodule case, or the operation ${}_\bullet\!\dgal{-,-}_\bullet$ in the outer bimodule case. 

\item (\emph{Abelianization}.) 
With the right bimodule, any double bracket on an algebra $A$ descends to its abelianization $A^{\ab}$ by Lemma \ref{Lem:Dbr-Ab-R}. 
This does not hold for the outer bimodule except in trivial cases.

\item (\emph{Example of application of }1\! \&\! 3.) 
Fix $r\geq2$ and let $\mathcal{P}_r=\kk[x_1,\ldots,x_r]$ be the commutative polynomial ring in $r$ variables. 
With the outer bimodule, it was shown by Powell \cite{P16} that the only double Poisson bracket on $\mathcal{P}_r$ is the zero double bracket.  
With the right bimodule, any double $(12)$-weak Poisson bracket on $\kk\langle x_1,\ldots,x_r\rangle$ descends to a double $(12)$-weak Poisson bracket on $\mathcal{P}_r$. 
As a non-trivial family of examples, any choice of constants $(\lambda_{ij})_{1\leq i,j\leq r}$ with $\lambda_{ij}=-\lambda_{ji}$ defines a double $(12)$-weak Poisson bracket by setting $\dgal{x_i,x_j}=\lambda_{ij}\,1\otimes 1$ thanks to Lemma \ref{Lem:wJac2}. 

\item (\emph{Geometric differences}.) 
In both the outer and right bimodule cases, we can obtain an antisymmetric bilinear operation on representation spaces thanks to Propositions \ref{Pr:Rep-twOut} and \ref{Pr:Rep-twR}. 
However, these operations on representation spaces are defined differently as can be seen from \eqref{Eq:Rep-twOut} and \eqref{Eq:Rep-twR}.

\item (\emph{Geometric differences (bis)}.) 
At the level of representation spaces, we can define on matrix-valued functions of the form $\Xtt(a)$ the operation $\dgal{-,-}_{(n)}$ from Remark \ref{Rem:RepMat-Out} with the outer bimodule, or the operation 
$\br{-\stackrel{\otimes}{,}-}$ from Remark \ref{Rem:RepMat-Right} with the right bimodule.  
As noticed in the Introduction, we can go back and forth between one notation or the other. However, both operations will \emph{not} necessarily be induced by a corresponding double bracket. 

\item (\emph{Example of application of }4\! \&\! 6.) 
There is a double $(12)$-weak Poisson bracket on $\mathcal{P}_2$ associated with the right bimodule structure which is determined by 
$\dgal{x,x}=0=\dgal{y,y}$ and $\dgal{x,y}= 1\otimes 1$. 
It gives rise to a unique Poisson bracket on $\Rep(\mathcal{P}_2,n)$, $n \geq 1$, by Theorem \ref{Thm:Rep-Right-Lie}. Using the notation of Remark \ref{Rem:RepMat-Right}, the Poisson structure on $\Rep(\mathcal{P}_2,n)$ is such that  
$$\br{\Xtt(x)\stackrel{\otimes}{,}\Xtt(x)}=0\,, \quad \br{\Xtt(y)\stackrel{\otimes}{,}\Xtt(y)}=0\,,\quad 
\br{\Xtt(x)\stackrel{\otimes}{,}\Xtt(y)}= \Id_n \otimes \Id_n\,.$$
Alternatively, it can be encoded into the operation $\dgal{-,-}_{(n)}$ \eqref{Eq:Br-VdB} satisfying the rules \eqref{Eq:Intro-VdB}:  
\begin{equation} \label{Eq:Rep-Comp}
    \dgal{\Xtt(x),\Xtt(x)}_{(n)}=0\,, \quad 
\dgal{\Xtt(y),\Xtt(y)}_{(n)}=0\,, \quad 
\dgal{\Xtt(x),\Xtt(y)}_{(n)}=\sum_{1\leq i,j \leq n} E_{ij}\otimes E_{ji}\,.
\end{equation}
There is no nonzero double Poisson bracket on $\mathcal{P}_2$, so the operation $\dgal{-,-}_{(n)}$ satisfying \eqref{Eq:Rep-Comp} is \emph{not} induced by a double Poisson bracket (after combining Proposition \ref{Pr:Rep-twOut} and Remark \ref{Rem:RepMat-Out}). 
\end{enumerate}


\appendix

\section{Case study: Gradient double Poisson algebras}  \label{S:Case}

Consider the \emph{commutative} polynomial algebra in $3$ variables $\mathcal{P}_3=\kk[x_1,x_2,x_3]$, and for $1\leq j \leq 3$ denote by $\partial_j$ the standard derivation on $\mathcal{P}_3$ satisfying $\partial_j(x_k)=\delta_{jk}$.  
Denote by $(\epsilon^{ijk})_{1\leq i,j,k\leq 3}$ the totally antisymmetric $3$-tensor satisfying $\epsilon^{123}=1$. 
 Fix an element $\varphi\in \mathcal{P}_3$. 
It is well-known (see e.g. \cite[\S9.2]{LGPV}) that we have a Poisson bracket $\br{-,-}_{\varphi}$ on $\mathcal{P}_3$ which is given on generators through 
\begin{equation}
\br{x_i,x_j}_{\varphi}=\sum_{1\leq k \leq 3} \epsilon^{ijk} \partial_k(\varphi)\,, 
\quad 1\leq i,j \leq 3\,.
\end{equation}
Furthermore, $\varphi$ is a Casimir of the Poisson bracket, that is $\br{\varphi,-}_{\varphi}$ acts as the zero derivation on $\mathcal{P}_3$. Such a Poisson bracket is called \emph{gradient}, since the components of the corresponding bivector in the basis $\{\partial_2\wedge \partial_3,\partial_3\wedge \partial_1,\partial_1\wedge \partial_2\}$ are given by $\{\partial_1(\varphi),\partial_2(\varphi),\partial_3(\varphi)\}$.

In the rest of this appendix, we  start a classification of double Poisson brackets associated with the outer bimodule structure (as originally introduced by Van den Bergh \cite{VdB1}, see \ref{ss:E-Out}) which generalise gradient Poisson brackets. 
This illustrates the differences between (commutative) Poisson brackets and (associative) double Poisson brackets for readers who have a limited knowledge of the latter. 

\subsection{Conventions and main results}

We set $A=\kk\langle x_1,x_2,x_3\rangle$. 
We equip $A$ with a grading $|-|$ by setting $|x_j|=1$ for $1\leq j \leq 3$ and $|\kk|=0$. Hence, a homogeneous element $f\in A$ of degree $d$ is of the form 
\begin{equation} \label{Eq:f-decomp}
f=\sum_I \gamma_I \, x_{i_1}\ldots x_{i_d},\quad 
\gamma_I \in \kk, \quad I=(i_1,\ldots,i_d)\in \{1,2,3\}^{\times d}\, .
\end{equation}
In that case, we say that (the homogeneous element of degree $d$) $f\in A$ is \emph{fully non-commutative} if it is $S_d$-invariant: given the decomposition \eqref{Eq:f-decomp}, we can also write for any $\sigma \in S_d$ that 
\begin{equation} \label{Eq:f-decomp-bis}
f=\sum_I \gamma_I \, x_{i_{\sigma(1)}}\ldots x_{i_{\sigma(d)}}\,.
\end{equation} 
More generally, we can decompose any nonzero $f\in A$ as $f=\sum_{w=0}^d f_w$ where  $d\geq 0$ and $|f_w|=w$. We then say that $f\in A$ is \emph{fully non-commutative} if each homogeneous part $f_w$, $0\leq w \leq d$, is fully non-commutative. 
We note that the grading extends inductively from $A$ to $A^{\otimes n}$, $n \geq 2$, by setting $|f\otimes g|=|f|+|g|$ if $f\in A$ and $g\in A^{\otimes (n-1)}$ are homogeneous. 

We endow $A^{\otimes 2}$ with the outer $\cdot := \cdot_{out}$ and inner $\ast := \cdot_{in}$ bimodule structures given in \eqref{Eq:Mod-out}--\eqref{Eq:Mod-in}. 
The additive group of double derivations on $A$ is defined as 
$$\DDer(A)=\{\delta \in \Hom_\kk(A,A^{\otimes 2}) \mid \delta(ab)=a\cdot \delta(b)+\delta(a)\cdot b\}\,.$$
For any $1\leq j \leq 3$, we consider the double derivation $\delh_j\in \DDer(A)$ uniquely defined by  
\begin{equation}
    \delh_j(x_k)=\delta_{jk} \, 1\otimes 1\,, \qquad 1\leq k \leq 3\,.
\end{equation}
We have that $\delh_j$ has degree $-1$ if we endow $A$ and $A^{\otimes 2}$ with the grading introduced above. 
Hereafter, by a double bracket we mean a double bracket associated with the outer bimodule structure on $A^{\otimes 2}$. In other words, a double bracket $\dgal{-,-}:A^{\times 2}\to A^{\otimes 2}$ is as in Definition \ref{Def:Dbr} with $\cdot$ and $\ast$ being the outer and inner bimodule structures on $A^{\otimes 2}$ respectively, as specified above. 

\begin{problem} \label{Prob:A}
Classify all the double Poisson brackets on $A$ which satisfy 
\begin{equation} \label{Eq:ProbA}
\dgal{x_i,x_j}_{f}=\sum_{1\leq k \leq 3} \epsilon^{ijk} \delh_k(f)\,, 
\quad 1\leq i,j \leq 3\,,
\end{equation}
for some $f\in A$. Such pairs $(A,\dgal{-,-}_f)$ are called \emph{gradient double Poisson algebras}. 
\end{problem}

Thanks to the choice of bimodule structures taken above, we can evaluate any double bracket satisfying \eqref{Eq:ProbA} on monomials: given $a=x_{i_1}\ldots x_{i_m}\in A$ and $b=x_{j_1}\ldots x_{j_n}\in A$, we have  
\begin{align*}
    \dgal{a,b}_f&= 
\sum_{\mu=1}^m \sum_{\nu=1}^n 
\sum_{1\leq k \leq 3} \epsilon^{i_\mu j_\nu k}\,\,
x_{i_1}\ldots x_{i_{\mu-1}} \ast x_{j_1}\ldots x_{j_{\nu-1}} \cdot 
\delh_k(f)
\cdot x_{j_{\nu+1}}\ldots x_{j_n}\ast x_{i_{\mu+1}}\ldots x_{i_m} \\
&=\sum_{\mu=1}^m \sum_{\nu=1}^n 
\sum_{1\leq k \leq 3} \epsilon^{i_\mu j_\nu k}\,\,
x_{j_1}\ldots x_{j_{\nu-1}} 
\delh_k(f)' 
x_{i_{\mu+1}}\ldots x_{i_m}
\otimes 
x_{i_1}\ldots x_{i_{\mu-1}}
\delh_k(f)''
x_{j_{\nu+1}}\ldots x_{j_n}  \,.
\end{align*}
(We follow the convention that an empty product equals to $+1$, e.g. $x_{i_1}\ldots x_{i_{\mu-1}}=1\in A$ when $\mu=1$.)
We directly obtain the double bracket of any two elements of $A$ by $\kk$-bilinearity.

Our progress on tackling Problem \ref{Prob:A} is presented as part of the following results, which are proven in \ref{ss:App-Pf}. As an application, we gather several (counter-)examples in \ref{ss:App-Ex}.  

\begin{proposition} \label{Pr:Prob-anti}
An element $f\in A$ is such that \eqref{Eq:ProbA} defines a double bracket if and only if it is fully non-commutative. 
\end{proposition}

\begin{lemma}\label{Lem:Prob-deg}
Decompose $f\in A\setminus \{0\}$ with respect to the grading on $A$ as $f=\sum_{w=0}^d f_w$ where $|f_w|=w$ and $d\geq 0$. If $\dgal{-,-}_f$ is a double Poisson bracket, then so is $\dgal{-,-}_{f_d}$. 
\end{lemma}
\begin{theorem} \label{Thm:ProbRes}
We have that $\dgal{-,-}_f$ is a double Poisson bracket if $f\in A$ is such that 
\begin{enumerate}
    \item[(i)] $f=x_j^d$ for some $j\in\{1,2,3\}$ and $d\geq 0$; 
    \item[(ii)] $f=(x_1+x_2+x_3)^d$ for some $d\geq 0$. 
\end{enumerate}
Furthermore, we can take linear combinations in each case, that is we also have a double Poisson bracket if $f=\sum_{w=0}^d \gamma_w x_j^w$ or if $f=\sum_{w=0}^d \gamma_w (x_1+x_2+x_3)^w$ with $\gamma_0,\ldots,\gamma_d \in \kk$. 
In particular, we have $\dgal{f,-}_f=0$ in all these cases. 
\end{theorem}

\subsection{Working on the classification} \label{ss:App-Pf}

\subsubsection{Proof of Proposition \ref{Pr:Prob-anti}}

We first note that the cyclic antisymmetry of the double bracket is equivalent to showing that for any $j,k=1,2,3$ distinct, we have $\dgal{x_j,x_k}_f=-\dgal{x_k,x_j}_f^\circ$. (Recall that the Leibniz rules guarantee that cyclic antisymmetry holds if it does on generators.)
Due to \eqref{Eq:ProbA}, these identities are equivalent to 
\begin{equation} \label{Eq:Prob-anti-1}
    \delh_i(f)=\delh_i(f)^\circ\,, \qquad 1\leq i \leq 3\,.
\end{equation} 
Thus, we have to check that \eqref{Eq:Prob-anti-1} holds if and only if $f$ is fully non-commutative. 
Furthermore, since each $\delh_i$ is a double derivation of degree $-1$, we can assume without loss of generality that $f$ is homogeneous of degree $d$. There is clearly nothing to check if $d=0$, i.e. $f\in  \kk$. 

Let us first assume that $\dgal{-,-}_f$ is a double bracket, i.e. \eqref{Eq:Prob-anti-1} holds. 
Fix an arbitrary term $\gamma \,x_{i_1}\ldots x_{i_d}$ in $f$, where $\gamma\in \kk^\times$. Applying $\delh_{i_w}$ for some $1\leq w \leq d$ to that term, we see that there is a term   $\gamma \,x_{i_1}\ldots x_{i_{w-1}}\otimes x_{i_{w+1}}\ldots x_{i_d}$ in $\delh_{i_w}(f)$. Due to \eqref{Eq:Prob-anti-1}, the term 
$\gamma \, x_{i_{w+1}}\ldots x_{i_d}\otimes x_{i_1}\ldots x_{i_{w-1}}$ must occur in $\delh_{i_w}(f)$, hence there is a term $\gamma \, x_{i_{w+1}}\ldots x_{i_d} x_{i_w} x_{i_1}\ldots x_{i_{w-1}}$ in $f$. In particular, the case $w=1$ (or $w=d$) amounts to a cyclic permutation of the factors. 
The case $w=2$ followed by the cyclic permutation amounts to swapping the first two factors. 
Therefore, by combining these two generators of the symmetric group, we can see that for any $\sigma \in S_d$ the  term 
 $\gamma \,x_{i_{\sigma(1)}}\ldots x_{i_{\sigma(d)}}$ appears in $f$. In particular, $f$ is fully non-commutative. 
 
Conversely, assume that $f$ is fully non-commutative. 
If a term $\gamma \,x_{i_1}\ldots x_{i_{w-1}}\otimes x_{i_{w+1}}\ldots x_{i_d}$, $\gamma\in \kk^\times$, appears in $\delh_j(f)$ we must have a term 
$\gamma \,x_{i_1}\ldots x_{i_{w-1}}x_j x_{i_{w+1}}\ldots x_{i_d}$ in $f$. Being fully non-commutative, 
$\gamma \,x_{i_{w+1}}\ldots x_{i_d}x_j x_{i_1}\ldots x_{i_{w-1}}$ is also a term in $f$. Thus $\gamma \,x_{i_{w+1}}\ldots x_{i_d}\otimes x_{i_1}\ldots x_{i_{w-1}}$ appears in $\delh_j(f)$, hence the term that we picked earlier is a term in $\delh_j(f)^\circ$.  

\subsubsection{Proof of Lemma \ref{Lem:Prob-deg}}

Using the grading of $A$ to decompose $f=\sum_{w=0}^d f_w$ in homogeneous components, $|f_w|=w$, 
we can split the double bracket $\dgal{-,-}_f$ as $\sum_{w=0}^d \dgal{-,-}_{f_w}$. 
By assumption, $\dgal{-,-}_f$ is a double bracket, so by Proposition \ref{Pr:Prob-anti} $f$ is fully non-commutative. Therefore each $f_w$ is fully non-commutative as well, and each $\dgal{-,-}_{f_w}$ is a double bracket again by Proposition \ref{Pr:Prob-anti}. 

Recall that $A^{\otimes 2}$ and $A^{\otimes 3}$ are naturally endowed with gradings such that $|x_j|=1$ and $|\kk|=0$, so that $\delh_j:A\to A^{\otimes 2}$ has degree $-1$ for each $1\leq j \leq 3$. 
Therefore, we have from \eqref{Eq:ProbA} that each double bracket $\dgal{-,-}_{f_w}$ is an operation of degree $w-3$ because $|\dgal{x_j,x_k}_{f_w}|=w-1$. 
When computing the double Jacobiator $\dgal{x_i,x_j,x_k}_f$ with $1\leq i,j,k \leq 3$, we see that the highest possible term in $A^{\otimes 3}$ is the one coming from $\dgal{x_i,x_j,x_k}_{f_d}$ in degree $2d-3$. Since $\dgal{-,-,-}_f=0$ by assumption, the operation of highest degree $\dgal{-,-,-}_{f_d}$ must vanish. Hence $\dgal{-,-}_{f_d}$ is a double Poisson bracket, as desired.

\subsubsection{Proof of Theorem \ref{Thm:ProbRes}}

We prove the different statements in the case \emph{(ii)}, and we leave the easier case \emph{(i)} to the reader. 
We set $f_d:=(x_1+x_2+x_3)^d$ for any $d\geq 0$ so that $f_d=f_1^d$. Note that $f_d$ is fully non-commutative, so it defines a double bracket $\dgal{-,-}_{f_d}$ by Proposition \ref{Pr:Prob-anti}. 
Our next step consists in checking that $\dgal{-,-}_{f_d}$ is a double Poisson bracket. 
We start with the following useful results. 

\begin{lemma} \label{Lem:PfApp1}
Fix $d\geq 0$. For any $a\in A$, $\dgal{f_1,a}_{f_d}=0$. 
In particular, 
$\dgal{f_p,a}_{f_d}=0$ for any $p\geq 0$. 
\end{lemma}
\begin{proof}
The first part is trivial if $d=0$ since the double bracket is the zero one. The second part is trivial if $p=0$ by linearity. We omit those instances and assume $d,p\geq 1$ below. 

We compute that for any $1\leq k \leq 3$, since $\delh_k \in \DDer(A)$,
\begin{equation} \label{Eq:DelhF}
    \delh_k(f_d)=\sum_{\delta=0}^{d-1} f_\delta \cdot \delh_k(f_1) \cdot f_{d-\delta-1}
    = \sum_{\delta=0}^{d-1} f_\delta \otimes f_{d-\delta-1} \,,
\end{equation}
which is independent of $k$. Thus, for any $1\leq j \leq 3$ we get from \eqref{Eq:ProbA} that 
\begin{align*}
    \dgal{f_1,x_j}_{f_d}=\sum_{1\leq i,j \leq 3} \epsilon^{ijk} \delh_k(f_d)
=\delh_{j+1}(f_d)-\delh_{j-1}(f_d) =0 \,.   
\end{align*}
(The indices $j\pm 1$ are taken modulo $3$.) By the right Leibniz rule \eqref{Eq:Dbr-Der1}, $\dgal{f_1,a}_{f_d}=0$ for any $a\in A$. 
We then get $\dgal{f_p,a}_{f_d}=0$ using the left Leibniz rule \eqref{Eq:Dbr-Der2} and $f_p=f_1^p$. 
\end{proof}

\begin{lemma} \label{Lem:PfApp2}
For any $1\leq i,j,k\leq 3$ and $d,p\geq 0$, we have 
$\dgal{x_i, \dgal{x_j,x_k}_{f_d}}_{f_p,L}=0$.
\end{lemma}
\begin{proof}
Using \eqref{Eq:dbr-L}, \eqref{Eq:ProbA} and \eqref{Eq:DelhF}, we write that 
\begin{align*}
\dgal{x_i, \dgal{x_j,x_k}_{f_d}}_{f_p,L}=&
\sum_{1\leq \ell \leq 3} \epsilon^{jk \ell}
\dgal{x_i, \delh_\ell(f_d)}_{f_p,L} 
=\sum_{1\leq \ell \leq 3} \epsilon^{jk \ell} \sum_{\delta=0}^{d-1} 
\dgal{x_i, f_\delta}_{f_p} \otimes f_{d-\delta -1}\,.
\end{align*}
This is zero because $\dgal{x_i, f_\delta}_{f_p}$ vanishes for all $1\leq \delta \leq d$ by Lemma \ref{Lem:PfApp1}. 
\end{proof}

We now check the vanishing of the double Jacobiator of $\dgal{-,-}_{f_d}$ on generators. By \eqref{Eq:dJac}, we have for $1\leq i,j,k\leq 3$ that 
\begin{align*}
    \dgal{x_i,x_j,x_k}_{f_d}=&
\sum_{s=0,1,2}\tau_{(123)}^s \circ 
\dgal{-, \dgal{-,-}_{f_d}}_{f_d,L}\circ \tau_{(123)}^{-s}(x_i,x_j,x_k)\,,
\end{align*}
and each term vanishes due to Lemma \ref{Lem:PfApp2}. 

Let us now assume that 
$f=\sum_{w=0}^d \gamma_w (x_1+x_2+x_3)^w=\sum_{w=0}^d \gamma_w f_w$  with $\gamma_0,\ldots,\gamma_{d-1} \in \kk$ and $\gamma_d\in \kk^\times$. The double bracket $\dgal{-,-}_f$ is Poisson as a consequence of Lemma \ref{Lem:PfApp2}. Indeed, the double Jacobiator vanishes on generators due to the decomposition  
\begin{align*}
    \dgal{-,-,-}_{f}=&
\sum_{s=0,1,2}\tau_{(123)}^s \circ 
\dgal{-, \dgal{-,-}_{f}}_{f,L}\circ \tau_{(123)}^{-s} \\
=&\sum_{w,w'=0}^d \gamma_w \gamma_{w'} 
\sum_{s=0,1,2}\tau_{(123)}^s \circ 
\dgal{-, \dgal{-,-}_{f_w}}_{f_{w'},L}\circ \tau_{(123)}^{-s} \,,
\end{align*}
which holds by bilinearity of a double bracket. 
In particular, we can also get that $\dgal{f,-}_f=0$ by bilinearity using Lemma \ref{Lem:PfApp1}.

\subsection{Applications and examples} \label{ss:App-Ex}

\begin{example}
The element $f=x_1x_2$ is not fully non-commutative. The operation $\dgal{-,-}_f:A^{\times 2}\to A^{\otimes 2}$ satisfying \eqref{Eq:ProbA} and the left/right Leibniz rules is \emph{not} a double bracket by Proposition \ref{Pr:Prob-anti}. This is readily checked: \eqref{Eq:ProbA} yields 
$\dgal{x_2,x_3}_f=1\otimes x_2$ and $\dgal{x_3,x_2}_f=-1\otimes x_2$, so that the cyclic antisymmetry \eqref{Eq:Dbr-cyc} fails. 
\end{example}

\begin{example}
Set $f=x_1^d$ for some $d\geq 1$. 
The operation $\dgal{-,-}_{f}$ obtained through \eqref{Eq:ProbA} is a double Poisson bracket by case \emph{(i)} of Theorem \ref{Thm:ProbRes}. Its explicit form on generators is 
$\dgal{x_1,x_j}_{f}=0$, $\dgal{x_j,x_j}_{f}=0$ for $1\leq j \leq 3$, together with 
$\dgal{x_2,x_3}_{f}=\sum_{\delta=0}^{d-1} x_1^\delta \otimes x_1^{d-\delta-1}$. 
Using Theorem \ref{Thm:ProbRes} again, any linear combination of these double Poisson brackets is again a double Poisson bracket. 
\end{example}

\begin{example}
Set $f=(x_1+x_2+x_3)^d$ for some $d\geq 1$. 
The operation $\dgal{-,-}_f$ satisfying \eqref{Eq:ProbA} is a double Poisson bracket by case \emph{(ii)} of Theorem \ref{Thm:ProbRes}. It is such that  
$\dgal{x_j,x_j}_{f}=0$ for $1\leq j \leq 3$, while for $1\leq j \leq 3$ (with $j+1$ taken modulo $3$)  
$$\dgal{x_j,x_{j+1}}_{f}=\sum_{\delta=0}^{d-1} (x_1+x_2+x_3)^\delta \otimes (x_1+x_2+x_3)^{d-\delta-1}\,.$$ 
Any linear combination of these double Poisson brackets also yields a double Poisson bracket. 
\end{example}

The previous two cases coming from Theorem \ref{Thm:ProbRes} do not exhaust the full list of solutions to Problem \ref{Prob:A}: we can also consider linear polynomials, as we show in Example \ref{Ex:Linear}. While the classification could still be incomplete, we expect that a generic fully non-commutative polynomial will fail to define a double Poisson bracket. In particular, we believe that fully non-commutative polynomials in two variables (of highest degree $d>1$) will fail to yield double Poisson brackets, and some instances are gathered in Examples \ref{Ex:Quad-Prob} and \ref{Ex:HQuad-Prob}.

\begin{example} \label{Ex:Linear}
Fix $\zeta_k\in \kk$ for $0\leq k \leq 3$. Let $f=\sum_{1\leq k\leq 3} \zeta_k x_k + \zeta_0$, which is fully non-commutative. 
Then the operation $\dgal{-,-}_f$ obtained through \eqref{Eq:ProbA} is a double bracket by Proposition \ref{Pr:Prob-anti}, which is Poisson. 
Indeed, on generators we can write 
$\dgal{x_j,x_j}_f=0$ for $1\leq j \leq 3$ and 
$$\dgal{x_1,x_2}_f=\zeta_3\, 1\otimes 1,\quad \dgal{x_2,x_3}_f=\zeta_1\, 1\otimes 1, \quad 
\dgal{x_3,x_1}_f=\zeta_2\, 1\otimes 1\,.$$
This is a constant double bracket, so it is Poisson by Lemma \ref{Lem:Jac2} (the assumptions are satisfied because the double Jacobiator is a derivation in each argument thanks to Lemma \ref{Lem:OutJac}).
\end{example}

\begin{example} \label{Ex:Quad-Prob}
The element $f=x_1x_2+x_2x_1$ is fully non-commutative. Hence, it defines a double bracket through \eqref{Eq:ProbA} by Proposition \ref{Pr:Prob-anti}, which is such that $\dgal{x_j,x_j}_f=0$ for $1\leq j \leq 3$ and 
$$\dgal{x_1,x_2}_f=0,\quad \dgal{x_2,x_3}_f=1\otimes x_2+x_2\otimes 1, \quad \dgal{x_3,x_1}_f=1\otimes x_1+x_1\otimes 1\,.$$
The double bracket is not Poisson, i.e. the double Jacobiator \eqref{Eq:dJac} does not vanish identically. Indeed, 
\begin{align*}
    \dgal{x_3,x_2,x_3}_f
=& \dgal{x_3,\dgal{x_2,x_3}_f}_{f,L}+\tau_{(132)} \dgal{x_3,\dgal{x_3,x_2}_f}_{f,L} \\
=&\dgal{x_3,x_2}_f \otimes 1 
- \tau_{(132)} \dgal{x_3,x_2}_f \otimes 1 \\
=&1\otimes 1 \otimes x_2 - 1 \otimes x_2 \otimes 1 \neq 0\,.
\end{align*}
\end{example}

\begin{example} \label{Ex:HQuad-Prob}
Fix a degree $d\geq 2$ and some $1\leq e\leq d-1$. 
The fully non-commutative polynomial $f\in \kk\langle x_1,x_2\rangle\subset A$ of degree $d$ with $e$ factors $x_1$ is given by 
\begin{equation*}
f=\sum_{\sigma\in S_d} x_{i_{\sigma(1)}}\ldots x_{i_{\sigma(d)}}  \,, \quad 
i_1=\ldots=i_{e}=1,\quad  i_{e+1}=\ldots=i_d=2\,.
\end{equation*}
(The case $d=2$ is given in Example \ref{Ex:Quad-Prob}.) 
We claim that any fully non-commutative $g\in A$ with term of highest degree equal to $f$ does not define a double Poisson bracket through \eqref{Eq:ProbA}; by Lemma \ref{Lem:Prob-deg}, it suffices to show that the double bracket $\dgal{-,-}_f$ associated with $f$ is not Poisson. 

Our aim is to show that $\dgal{x_2,x_3,x_3}_f\neq 0$. 
Since $\dgal{x_3,x_3}_f=0$ and $\dgal{x_2,x_3}_f=\delh_1(f)$, we can see from \eqref{Eq:dJac} that this boils down to proving 
\begin{equation}
    (1-\tau_{(123)}) \dgal{x_3,\delh_1(f)}_{f,L} \neq 0\,.
\end{equation}
In other words, we seek a nonzero term in $A^{\otimes 3}$ from this expression. 
Our precise aim consists in finding, among all terms of highest degree in $A\otimes \kk \otimes \kk$, a unique nonzero term of the form $x_1^\alpha x_2^\beta x_1^\gamma\otimes 1 \otimes 1$ with $\beta> 0$, $\alpha\geq \gamma \geq 0$, which has highest order $\alpha$.  

Introduce the symmetrising map $C^w:\kk\langle x_1,x_2\rangle_w\to \kk\langle x_1,x_2\rangle_w$, where $\kk\langle x_1,x_2\rangle_w\subset \kk\langle x_1,x_2\rangle$ is the sub-vector space of degree $w$ elements, acting as 
$C^w(x_{j_1}\ldots x_{j_w})=
\sum_{\sigma\in S_w} x_{j_{\sigma(1)}}\ldots x_{j_{\sigma(w)}}$, for $j_1,\ldots,j_w\in \{1,2\}$. 
Note in particular that $f=C^d(x_1^e x_2^{d-e})$. 
Furthermore, we note that the terms of $\delh_1(f)$ of highest degree in $A\otimes \kk\subset A^{\otimes 2}$ are $C^{d-1}(x_1^{e-1}x_2^{d-e})\otimes 1$. 
In $\{\!\{x_3,\delh_1(f)\}\!\}_{f,L}$, these terms yield\footnote{The first term in this expression occurs only if $e\geq 2$. We write it also for $e=1$ since it does not provide the maximal term in the following discussion.
The second term occurs if $d-e\geq 1$, which is in our assumptions.} after restriction to $A\otimes \kk\otimes \kk$ 
\begin{align*}
    &C^{d-2}(x_1^{e-2}x_2^{d-e})C^{d-1}(x_1^e x_2^{d-e-1})\otimes 1\otimes 1 \\
-&C^{d-2}(x_1^{e-1}x_2^{d-e-1})C^{d-1}(x_1^{e-1} x_2^{d-e})\otimes 1\otimes 1 \,.
\end{align*}
Since $d-e\geq1$, there is a unique term of the desired form which is 
$-x_1^{e-1}x_2^{2d-2e-1}x_1^{e-1}\otimes 1\otimes 1$ if $d-e-1\geq 1$, or 
$-x_1^{2e-2}x_2^{d-e}\otimes 1 \otimes 1$ if $d-e=1$. 

Meanwhile, terms in $A\otimes \kk\otimes \kk$ coming from $-\tau_{(123)}\{\!\{x_3,\delh_1(f)\}\!\}_{f,L}$ are terms in $\kk\otimes\kk\otimes A$ of $-\{\!\{x_3,\delh_1(f)\}\!\}_{f,L}$. By an argument similar to the one above, the terms of $-\{\!\{x_3,\delh_1(f)\}\!\}_{f,L}$ with highest order with respect to the third tensor factor are  
\begin{align*}
&-\{\!\{x_3,x_1\otimes C^{d-2}(x_1^{e-2}x_2^{d-e})\}\!\}_{f,L}
-\{\!\{x_3,x_2\otimes C^{d-2}(x_1^{e-1}x_2^{d-e-1})\}\!\}_{f,L} \\
=&-\delh_2(f)\otimes C^{d-2}(x_1^{e-2}x_2^{d-e})
+\delh_1(f)\otimes C^{d-2}(x_1^{e-1}x_2^{d-e-1})\,.
\end{align*}
Since $f$ has degree $d>1$, $\delh_1(f),\delh_2(f)$ are not contained in $\kk\otimes \kk$ so there is no term in $A\otimes \kk\otimes \kk$ coming from those. 
If we consider terms of $-\{\!\{x_3,\delh_1(f)\}\!\}_{f,L}$ with a lower order  with respect to the third tensor factor, they must  contain at least a factor $x_1$ or $x_2$ in one of the first two tensor factors. Therefore, $-\tau_{(123)}\{\!\{x_3,\delh_1(f)\}\!\}_{f,L}$ does not contain  terms of $A\otimes \kk\otimes \kk$. 
Hence, the unique term found above does not vanish in 
$(1-\tau_{(123)})\{\!\{x_3,\delh_1(f)\}\!\}_{f,L}$, as desired. 
\end{example}


 \addcontentsline{toc}{section}{References}

\Addresses


\begin{thebibliography}{99}
\bibitem{AKKN} Alekseev, A.; Kawazumi, N.; Kuno, Y.; Naef, F.: {\it Goldman-Turaev formality implies Kashiwara-Vergne}. Quantum Topol. 11, no. 4, 657--689 (2020); 
\href{https://arxiv.org/abs/1812.01159}{arXiv:1812.01159}.

\bibitem{Art} Arthamonov, S.: {\it Modified double Poisson brackets}. J. Algebra 492, 212--233 (2017);
\href{https://arxiv.org/abs/1608.08287}{arXiv:1608.08287}.

\bibitem{AOS} Arthamonov, S.; Ovenhouse, N.; Shapiro, M.: {\it Noncommutative networks on a cylinder}. Preprint; 
\href{https://arxiv.org/abs/2008.02889}{arXiv:2008.02889}. 

\bibitem{BBT} Babelon, O.; Bernard, D.; Talon, M.: {\it Introduction to classical integrable systems}.
Cambridge Monographs on Mathematical Physics. Cambridge University Press, Cambridge, 2003. 

\bibitem{B} Bielawski, R.: {\it Quivers and Poisson structures}. Manuscripta Math. 141, no. 1-2, 29–49 (2013); \href{https://arxiv.org/abs/1108.3222}{arXiv:1108.3222}.

\bibitem{BCS} Bozec, T.; Calaque, D.; Scherotzke, S.: {\it Calabi–Yau structures on (quasi-)bisymplectic algebras}. Preprint; 
 \href{https://arxiv.org/abs/2203.14382}{arXiv:2203.14382}.

\bibitem{CW} Casati, M.; Wang, J.P.: {\it Hamiltonian structures for integrable nonabelian difference equations}. Preprint; 
 \href{https://arxiv.org/abs/2101.06191}{arXiv:2101.06191}.
 
\bibitem{CEEY} Chen, X.; Eshmatov, A.; Eshmatov, F.; Yang, S.: {\it The derived non-commutative Poisson bracket on Koszul Calabi-Yau algebras}. J. Noncommut. Geom. 11, no. 1, 111--160  (2017);  
\href{https://arxiv.org/abs/1504.02885}{arXiv:1504.02885}.
 
\bibitem{CB11} Crawley-Boevey, W.: {\it Poisson structures on moduli spaces of representations}. J. Algebra 325, 205--215  (2011).

\bibitem{DSKV}  De Sole, A.; Kac, V.G.; Valeri, D.: {\it Double Poisson vertex algebras and non-commutative Hamiltonian equations}. Adv. Math. 281, 1025--1099 (2015); 
\href{https://arxiv.org/abs/1410.3325}{arXiv:1410.3325}.

\bibitem{F22} Fairon, M.: {\it Morphisms of double (quasi-)Poisson algebras and action-angle duality of integrable systems}.
Ann. Henri Lebesgue 5, 179--262 (2022); 
\href{https://arxiv.org/abs/2008.01409}{arXiv:2008.01409}.

\bibitem{FV} Fairon, M.; Valeri, D.: {\it Double multiplicative Poisson vertex algebras}. Preprint; 
\href{https://arxiv.org/abs/2110.03418}{arXiv:2110.03418}.


\bibitem{FH} Fern\'{a}ndez, D.; Herscovich, E.: {\it Cyclic $A_\infty$-algebras and double Poisson algebras}. J. Noncommut. Geom. 15, no. 1, 241--278 (2021); \href{https://arxiv.org/abs/1902.00787}{arXiv:1902.00787}. 

\bibitem{GG} Goncharov, M.; Gubarev, V.: {\it Double Lie algebras of a nonzero weight}. Preprint; \href{https://arxiv.org/abs/2104.13678}{arXiv:2104.13678}.

\bibitem{IKV} Iyudu, N.; Kontsevich, M.; Vlassopoulos, Y.: {\it Pre-Calabi-Yau algebras as noncommutative Poisson structures}. J. Algebra 567 (2021), 63--90; \href{https://arxiv.org/abs/1906.07134}{arXiv:1906.07134}. 

\bibitem{LGPV} Laurent-Gengoux, C.; Pichereau, A.; Vanhaecke, P.: {\it Poisson structures}. Grundlehren der mathematischen Wissenschaften, 347. Springer, Heidelberg, 2013. 

\bibitem{LP90}  Le Bruyn, L.; Procesi, C.: {\it Semisimple representations of quivers}. Trans. Amer. Math. Soc. 317, no. 2, 585--598 (1990).

\bibitem{L} Leray, J.: {\it Protoperads II: Koszul duality}. J. Éc. polytech. Math. 7 (2020), 897--941; \href{https://arxiv.org/abs/1901.05654}{arXiv:1901.05654}.

\bibitem{LV} Leray, J.; Vallette, B.: {\it Pre-Calabi--Yau algebras and homotopy double Poisson gebras}. Preprint;  \href{https://arxiv.org/abs/2203.05062}{arXiv:2203.05062}.

\bibitem{Lo} Loday, J.-L.: 
{\it Une version non commutative des algèbres de Lie: les algèbres de Leibniz}.  Enseign. Math. (2) 39, no. 3-4, 269--293 (1993).

\bibitem{ORS} Odesskii, A.V.; Rubtsov, V.N.; Sokolov, V.V.: {\it Double Poisson brackets on free associative algebras}. Noncommutative birational geometry, representations and combinatorics, Contemp. Math., vol. 592, Amer. Math. Soc., Providence, RI, 2013, pp. 225--239; 
\href{https://arxiv.org/abs/1208.2935}{arXiv:1208.2935}.

 \bibitem{ORS2} Odesskii, A.V.; Rubtsov, V.N.; Sokolov, V.V.: {\it Parameter-dependent associative Yang-Baxter equations and Poisson brackets}. Int. J. Geom. Methods Mod. Phys. 11, no. 9, 1460036, 18 pages (2014); 
 \href{https://arxiv.org/abs/1311.4321}{arXiv:1311.4321}.

\bibitem{PV} Pichereau, A.; Van de Weyer, G.: {\it Double Poisson cohomology of path algebras of quivers}. J. Algebra 319, no. 5, 2166--2208 (2008); 
\href{https://arxiv.org/abs/math/0701837}{arXiv:math/0701837}.

\bibitem{P16}  Powell, G.: {\it On double Poisson structures on commutative algebras}. J. Geom. Phys. 110, 1--8  (2016);  \href{https://arxiv.org/abs/1603.07553}{arXiv:1603.07553}.

\bibitem{Sch}  Schedler, T.: {\it Poisson algebras and Yang-Baxter equations}. In: Advances in quantum computation, 91–106, Contemp. Math., 482, Amer. Math. Soc., Providence, RI, 2009.

\bibitem{VdB1} Van den Bergh, M.: {\it Double Poisson algebras}. Trans. Amer. Math. Soc., 360, no. 11, 5711--5769 (2008);
\href{https://arxiv.org/abs/math/0410528}{arXiv:math/0410528}.

\bibitem{VdW} Van de Weyer, G.: {\it Double Poisson structures on finite dimensional semi-simple algebras}. Algebr. Represent. Theory 11, no. 5, 437--460 (2008); 
\href{https://arxiv.org/abs/math/0603533}{arXiv:math/0603533}.
\end{thebibliography}
\end{document}